\documentclass[journal]{IEEEtran}%
\usepackage{graphicx}
\usepackage{amsmath}
\usepackage{graphics}
\usepackage{amsfonts}
\usepackage{amssymb}%
\usepackage{listings}
\usepackage{amsthm}
\usepackage{epsfig}
\usepackage{psfrag}
\usepackage{booktabs}
\providecommand{\U}[1]{\protect\rule{.1in}{.1in}}

\newtheorem{thm}{Theorem}
\newtheorem{remark}{Remark}
\newtheorem{lemma}{Lemma}
\newtheorem{definition}{Definition}
\newtheorem{prop}{Proposition}
\newtheorem{coro}{Corollary}

\ifCLASSINFOpdf
\else
\fi
\begin{document}

\title{Linear Convergence of Adaptively Iterative Thresholding Algorithms for Compressed Sensing}
\author{Yu Wang, ~Jinshan Zeng$^{\ast}$, ~Zhimin Peng, ~Xiangyu Chang, ~and Zongben Xu
\thanks{ Y. Wang, J.S. Zeng and Z.B. Xu are with the Institute for Information and System
Sciences, School of Mathematics and Statistics, Xi'an Jiaotong University,
Xi'an 710049 (email: shifwang@gmail.com, jsh.zeng@gmail.com,zbxu@mail.xjtu.edu.cn).

Z.M. Peng is with the Department of Mathematics, University of California, Los Angeles (UCLA), Los Angeles, CA 90095, United States (email: zhimin.peng@math.ucla.edu).

X.Y. Chang is with the School of Management, Xi'an Jiaotong University, Xi'an 710049.
(email: xiangyuchang@gmail.com).

$*$ Corresponding author: Jinshan Zeng (jsh.zeng@gmail.com). }}
\maketitle

\begin{abstract}
This paper studies the convergence of the adaptively iterative thresholding (AIT) algorithm for compressed sensing.
We first introduce a generalized restricted isometry property (gRIP). Then we prove that the AIT algorithm converges to the original sparse solution at a linear rate under a certain gRIP condition in the noise free case. While in the noisy case, its convergence rate is also linear until attaining a certain error bound.
Moreover, as by-products, we also provide some sufficient conditions for the convergence of the AIT algorithm based on
the two well-known properties, i.e., the coherence  property and the restricted isometry property (RIP), respectively.
It should be pointed out that such two properties are special cases of gRIP.
The solid improvements on the theoretical results are demonstrated and compared with the known results.
Finally, we provide a series of simulations to verify the correctness of the theoretical assertions as well as the effectiveness of the AIT algorithm.
\newline

\end{abstract}




\markboth{ ~}{Shell \MakeLowercase{\textit{et al.}}: Bare Demo of IEEEtran.cls for Journals}






\IEEEpeerreviewmaketitle

\section{Introduction}


Let  $A\in \mathbf{R}^{m \times n}$, $b\in \mathbf{R}^m$ and $x \in \mathbf{R}^n$. Compressed sensing {\cite{Donoho06}}, {\cite{Candes06}} solves the following constrained $\ell_0$-minimization problem
\begin{equation}
\min_{x\in \mathbf{R}^n} {\|x\|_0} \ \ \text{s.t.}\ b=A x + \epsilon,\ \|\epsilon\|_2\leq \sigma
\label{L0MinExact}
\end{equation}
where $\epsilon \in \mathbf{R}^m$ is the measurement noise, $\sigma\in \mathbf{R}$ is the noise variance and $\|x\|_0$ denotes the number of the nonzero components of $x$. Due to the NP-hardness of problem (\ref{L0MinExact}) \cite{MallatMP1993}, approximate methods including the greedy method and relaxed method are introduced. The greedy method approaches the sparse solution by successively alternating one or more components
that yield the greatest improvement in quality {\cite{MallatMP1993}}.
These algorithms include iterative hard thresholding (IHT) \cite{FoucartRIP2010}, accelerated hard thresholding (AHT) \cite{ Cevher2011}, ALPS \cite{kyrillidis2011recipes}, hard thresholding pursuit (HTP) \cite{foucart2011hard}, CLASH \cite{kyrillidis2012combinatorial}, OMP {\cite{PatiOMP1993}}, {\cite{TroppOMP2007}}, StOMP {\cite{DonohoStOMP}}, ROMP {\cite{Needell2010ROMP}}, CoSaMP {\cite{CoSaMP}} and SP {\cite{DaiSP}}.
The greedy algorithms can be quite efficient and fast in many applications, especially when the signal is very sparse.

The relaxed method converts the combinatorial $\ell_0$-minimization into a more tractable model through replacing the $\ell_0$ norm with a nonnegative and continuous function $P(\cdot)$, that is,
\begin{equation}
\min_{x\in \mathbf{R}^n} {P(x)} \ \ \text{s.t.}\ b=A x + \epsilon, \ \|\epsilon\|_2\leq \sigma.
\label{PMinExact}
\end{equation}
One of the most important cases is the $\ell_1$-minimization problem
(also known as \textit{basis pursuit} (BP)) {\cite{Donoho98}} in the noise free case and \textit{basis pursuit denoising} in the noisy case)
with $P(x) = \|x\|_1$,
where $\|x\|_1 = \sum_{i=1}^n |x_i|$ is called the $\ell_1$ norm.
The $\ell_1$-minimization problem is a convex optimization problem that can be efficiently solved.
Nevertheless, the $\ell_1$ norm may not induce further sparsity when applied to certain applications
{\cite{Chartrand2007}}, {\cite{Chartrand2008}}, {\cite{L1/2TNN}}, {\cite{Candes2008RL1}}.
Therefore, many nonconvex functions were proposed as substitutions of the $\ell_0$ norm.
Some typical nonconvex examples include the $\ell_q$ ($0<q<1$) norm {\cite{Chartrand2007}}, {\cite{Chartrand2008}}, {\cite{L1/2TNN}},
smoothly clipped absolute deviation (SCAD) {\cite{FanSCAD}} and minimax concave penalty (MCP) {\cite{ZhangMCP2010}}.
Compared with the $\ell_1$-minimization model, the nonconvex relaxed models can often induce better sparsity and reduce the bias,
while they are generally more difficult to solve.

The iterative reweighted method and regularization method are two main classes of algorithms to solve (\ref{PMinExact}) when $P(x)$ is nonconvex.
The iterative reweighted method includes the iterative reweighted least squares minimization (IRLS) {\cite{FOCUSS}}, {\cite{DaubechiesIRLS}},
and the iterative reweighted $\ell_1$-minimization (IRL1) algorithms {\cite{Candes2008RL1}}.
Specifically, the IRLS algorithm solves a sequence of weighted least squares problems, which can be viewed as some approximations to the original optimization problem.
Similarly, the IRL1 algorithm solves a sequence of non-smooth weighted $\ell_1$-minimization problems,
and hence it is the non-smooth counterpart to the IRLS algorithm.
However, the iterative reweighted algorithms are slow if the nonconvex penalty
cannot be well approximated by the quadratic function or the weighted $\ell_1$ norm function.
The regularization method transforms problem (\ref{PMinExact}) into the following unconstrained optimization problem
\begin{equation}
\min_{x\in \mathbf{R}^n} \{\|Ax - b\|_2^2+\lambda P(x)\},
\label{PRegProb}
\end{equation}
where $\lambda>0$ is a regularization parameter.
For some special penalties $P(x)$ such as the $\ell_q$ norms ($q=0,1/2,2/3, 1$), SCAD and MCP,
an optimal solution of the model (\ref{PRegProb}) is a fixed point of the following equation
$$x=H(x-sA^T(Ax - b)),$$
where $H(\cdot)$ is a componentwise thresholding operator which will be defined in detail in the next section and $s>0$ is a step size parameter. This yields the corresponding iterative thresholding algorithm ({\cite{L1/2TNN}}, {\cite{DaubechiesSoft04}}, {\cite{L2/3Cao2013}}, {\cite{BlumensathHard08}}, {\cite{ZengConvHalf2014}}, {\cite{ZengConvJumping2014}})
$$x^{(t+1)}=H(x^{(t)}-sA^T(Ax^{(t)}-b)).$$
Compared to greedy methods and iterative reweighted algorithms,
iterative thresholding algorithms have relatively lower computational complexities {\cite{Qian20011}}, {\cite{Zeng20012SAR}}, {\cite{Zeng20013AccSAR}}.
So far, most of theoretical guarantees of the iterative thresholding algorithms were developed for the regularization model (\ref{PRegProb}) with fixed $\lambda$.
However, it is in general difficult to determine an appropriate regularization parameter $\lambda$.

Some adaptive strategies for setting the regularization parameters were proposed.
One strategy is to set the regularization parameter adaptively so that $\|x^{(t)}\|_0$ remains the same at each iteration.
This strategy was first applied to the iterative hard thresholding algorithm (called Hard algorithm for short henceforth) in {\cite{Blumensath08CS}},
and later the iterative soft thresholding algorithm{\cite{MalekiITA2009}} (called Soft algorithm for short henceforth) and the iterative half thresholding algorithm {\cite{L1/2TNN}} (called Half algorithm for short henceforth).
The convergence of  Hard algorithm was justified when $A$ satisfies the restricted isometry property (RIP) with $\delta_{3k^*}<\frac{1}{\sqrt{32}}$
{\cite{Blumensath08CS}}, where $k^*$ is the number of the nonzero components of the truely sparse signal.
Later, Maleki {\cite{MalekiITA2009}} investigated the convergence of both Hard and Soft algorithms in terms of the coherence.
Recently, Zeng et al. {\cite{ZengAIT2014}} generalized Maleki's results to a wide class of iterative thresholding algorithms.
However, most of guarantees in {\cite{ZengAIT2014}} are coherence-based and focus on the noise free case with the step size equal to 1.
While it has been observed that in practice, the AIT algorithm can have remarkable performances for noisy cases with a variety of step sizes.
In this paper, we develop the theoretical guarantees of the AIT algorithm with different step sizes in both noise free and noisy cases.

\subsection{Main Contributions}

The main contributions of this paper are the following.

\begin{enumerate}
\item[i)]
Based on the introduced gRIP, we give a new uniqueness theorem for the sparse signal (see Theorem {\ref{Unique_Thm}}), and then
show that the AIT algorithm can converge to the original sparse signal at a linear rate (See Theorem {\ref{conv_gRIP}}).
Specifically, in the noise free case, the AIT algorithm converges to the original sparse signal at a linear rate.
While in the noisy case, it also converges to the original sparse signal at a linear rate until reaching an error bound.

\item[ii)]
The tightness of our analyses is further discussed in two specific cases.
The coherence based condition for Soft algorithm
is the same as those required for both OMP and BP.
Moreover, the RIP based condition for Hard algorithm is $\delta_{3k^*+1}<\frac{\sqrt{5}-1}{2}\approx0.618$,
which is better than the results in \cite{foucart2011hard} and \cite{Foucart2013Introduction}.

\end{enumerate}

The rest of this paper is organized as follows.
In section II, we describe the adaptively iterative thresholding (AIT) algorithm.
In section III, we introduce the generalized restricted isometry property, and then provide a new uniqueness theorem.
In section IV, we prove the convergence of the AIT algorithm.
In section V, we compare the obtained theoretical results with some other known results.
In section VI, we implement a series of simulations to verify the correctness of the theoretical results as well as the efficiency of the AIT algorithm.
In section VII, we discuss many practical issues on the implementation of the AIT algorithm,
and then conclude this paper in section VIII.
All the proofs are presented in the Appendices.

{\bf Notations.}
We denote $\mathbf{N}$ and $\mathbf{R}$ as the natural number set and one-dimensional real space, respectively.
For any vector $x \in \mathbf{R}^n$, $x_i$ is the $i$-th component of $x$ for $i=1,\ldots, n.$
For any matrix $A \in \mathbf{R}^{m\times n}$, $A_i$ denotes the $i$-th column of $A$.
$x^T$ and $A^T$ represent the transpose of vector $x$ and matrix $A$ respectively.
For any index set $S\subset \{1,\ldots, n\} $, $|S|$ represents its cardinality.
$S^c$ is the complementary set, i.e., $S^c = \{1,\ldots, n\} \setminus S.$
For any vector $x\in \mathbf{R}^n$, $x_S$ represents the subvector of $x$ with the components restricted to $S$.
Similarly, $A_S$ represents the submatrix of $A$ with the columns restricted to $S$.
We denote $x^*$ as the original sparse signal with $\|x^*\|_0 = k^*$,
and $I^* = \{i: |x_i^*|\neq 0\}$ is the support set of $x^*.$
${\bf I}_r \in \mathbf{R}^{r\times r}$  is the $r$-dimensional identity matrix.
$sgn(\cdot)$ represents the signum function.

\section{Adaptively Iterative Thresholding Algorithm}

The AIT algorithm for (\ref{PRegProb})
is the following
\begin{equation}
z^{(t+1)} = x^{(t)} -  s A^{T}(A x^{(t)}-b),\\
\label{AIT1}
\end{equation}
\begin{equation}
x^{(t+1)} = H_{\tau^{(t+1)}}(z^{(t+1)}),
\label{AIT}
\end{equation}
where $s>0$ is a step size and
\begin{equation}
H_{\tau^{(t+1)}}(x) = (h_{\tau^{(t+1)}}(x_1),\ldots, h_{\tau^{(t+1)}}(x_n))^T
\end{equation}
is a componentwise thresholding operator. The thresholding function $h_{\tau}(u)$ is defined as
\begin{equation}
h_{\tau}(u)= \left\{
\begin{array}{cc}
f_{\tau}(u), & |u|> \tau \\
0, & {\rm otherwise}%
\end{array}%
\right.,
\label{ThreshFun}
\end{equation}%
where $f_{\tau}(u)$ is the \emph{defining function.}
In the following, we give some basic assumptions of the defining function, which were firstly introduced in {\cite{ZengAIT2014}}.

{\bf Assumption 1.} Assume that $f_{\tau}$ satisfies
\begin{enumerate}
  \item \textbf{Odevity}. $f_{\tau}(u)$ is an odd function of $u$.
  \item \textbf{Monotonicity}.  $f_{\tau}(u) < f_{\tau}(v)$ for any $\tau \leq u < v $.
  \item \textbf{Boundedness}.  There exist two constants $0\leq c_2\leq c_1\leq1$ such that $u - c_1 \tau \leq   f_{\tau}(u) \leq u-c_2\tau $  for $u\geq \tau$.
\end{enumerate}

Note that most of the commonly used thresholding functions satisfy Assumption 1.
In Fig. {\ref{Fig ThFun}}, we show some typical thresholding functions including hard {\cite{BlumensathHard08}},
soft {\cite{DaubechiesSoft04}} and half \cite{L1/2TNN} thresholding functions for $\ell_0, \ell_1, \ell_{1/2}$ norms respectively,
as well as the thresholding functions for $\ell_{2/3}$ norm \cite{L2/3Cao2013} and SCAD penalty {\cite{FanSCAD}}.
The corresponding boundedness parameters are shown in Table {\ref{TableBoundPar}}.

\begin{figure}
\centering
  \includegraphics[width=3in]{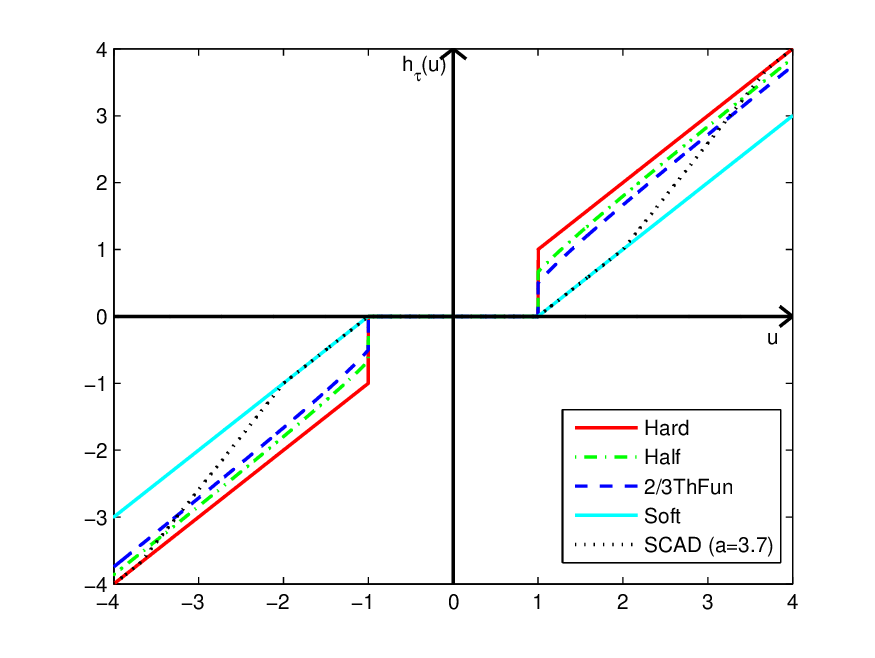}\\
  \caption{Typical thresholding functions $h_{\tau}(u)$ with $\tau=1$.}
  \label{Fig ThFun}
  \vspace*{-5pt}
\end{figure}

\begin{table} [h]
\centering
\caption{Boundedness parameters $c$ for different thersholding functions}
\label{TableBoundPar}
\begin{tabular}{cccccc}
\toprule
$f_{\tau,*}$   & $f_{\tau,0}$   & $f_{\tau,1/2}$     & $f_{\tau,2/3}$  & $f_{\tau,1}$  & $f_{\tau,SCAD}$   \\
\midrule
$c_1$            &  0        & $\frac{1}{3}$      & $\frac{1}{2} $     & 1             & 1  \\

$c_2$            &  0        & 0      & 0     & 1             & 0   \\
\bottomrule
\end{tabular}
\end{table}

This paper considers a heuristic way for setting the threshold $\tau^{(t)}$, specifically, we let
$$\tau^{(t)} =  |z^{(t)}_{[k+1]}|,$$
where $z^{(t)}_{[k+1]}$ is the $(k+1)$-th largest component of $z^{(t)}$ in magnitude and
$k$ is the \emph{specified sparsity level}, $[k+1]$ denotes the index of this component.
We formalise the AIT algorithm as in Algorithm 1.

\begin{table}[!htb]
\renewcommand{\arraystretch}{1.2}
\centering
\textbf{Algorithm 1: Adaptively Iterative Thresholding Algorithm}
\vspace{30pt}
\begin{tabular}{l}
\toprule
\textbf{\emph{Initialization}}:
{Normalize} $A$ such that $\|A_{j}\|_{2}=1$ for $j=1, \ldots, n$. \\ Given
a sparsity level $k$, a step size $s>0$ and an initial point $x^{(0)}$.\\
Let $t:=0$;  \\
\textbf{\emph{Step 1}}: Calculate $z^{(t+1)}=x^{(t)}-sA^{T}(Ax^{(t)}-b) $;\\
\textbf{\emph{Step 2}}: Set $\tau^{(t+1)} = |z^{(t+1)}_{[k+1]}|$ and $I^{t+1}$ as the index set of\\
 \ \ \ \ \ \ \ \ \ the largest $k$ components of $z^{(t+1)}$ in magnitude;\\

\textbf{\emph{Step 3}}: Update: if $i\in {I^{t+1}}$, $x^{(t+1)}_i=f_{\tau^{(t+1)}}(z^{(t+1)}_i)$ , otherwise \\
 \ \ \ \ \ \ \ \ \ \ $x^{(t+1)}_i=0$; \\
\textbf{\emph{Step 4}}: $t=t+1$ and repeat \textbf{\emph{Steps 1}}-\textbf{\emph{3}} until convergence.\\
\bottomrule
\end{tabular}
\vspace*{-40pt}
\end{table}

\begin{remark} \label{AIT_Alg}
At the $(t+1)$-th iteration,
the AIT algorithm yields a sparse vector $x^{(t+1)}$ with $k$ nonzero components.
The sparsity level $k$ is a crucial parameter for the performance of the AIT algorithm.  When $k\geq k^*$, the results will get better as $k$ decreases.
Once $k<k^*$, the AIT algorithm fails to find the original sparse solution.
Thus, $k$ should be specified as an upper bound estimate of $k^*$.
\end{remark}

\begin{remark}
In Algorithm 1, the columns of matrix $A$ are required to be normalized.
Such operation is only for a clearer definition of the following introduced generalized restricted isometry property (gRIP)
and more importantly, better theoretical analyses.
However, as shown in Section VII B, this requirement is generally not necessary
for the use of the AIT algorithm in the perspective of the recovery performance.
We will conduct a series of experiments in Section VII B for a detailed explanation.
\end{remark}

\section{Generalized Restricted Isometry Property}

This section introduces the generalized restricted isometry property (gRIP) and then gives the uniqueness theorem.

\begin{definition} For any matrix $A\in \mathbf{R}^{m\times n}$, and a constant pair $(p,q)$
where $p\in [1,\infty), q\in [1,\infty]$ and $\frac{1}{p}+\frac{1}{q}=1,$
then the $(k,p,q)$-generalized restricted isometry constant (gRIC) $\beta_{k,p,q}$ of $A$ is defined as
\begin{equation}
\beta_{k,p,q} = {\sup_{S\subset \{1,\ldots, n\}, \newline  |S|\leq k}} \ \sup_{x\in \mathbf{R}^{|S|}\setminus \{0\}} \frac{\|(\mathbf{I}_{|S|}-A_S^TA_S)x\|_{q}}{\|x\|_{p}}.
\label{Def_gRIP}
\end{equation}
\label{def_gRIP}
\end{definition}

We will show that the introduced gRIP satisfies the following proposition.

\begin{prop}
For any positive constant pair $(p,q)$ with $\frac{1}{p}+\frac{1}{q}=1$,
the generalized restricted isometric constant $\beta_{k,p,q}$ associated with $A$ and $k$ must satisfy
\begin{equation}
\frac{1}{3}\beta_{k,p,q}\leq
\sup_{z\in\mathbf{R}^n\setminus \{0\},\|z\|_0\leq k}\frac{\left|z^T(A^TA-\mathbf{I}_n)z\right|}{\|z\|_p^2}
\leq \beta_{k,p,q}.
\label{prop_gRIP_Exp}
\end{equation}
\label{prop_gRIP1}
\end{prop}

The proof of this proposition is presented in Appendix A.
It can be noted that the gRIP closely relates to the coherence property and restricted isometry property (RIP), whose definitions are listed in the following.
\begin{definition}
For any matrix $A\in \mathbf{R}^{m\times n}$,
the coherence of $A$ is defined as
\begin{equation}
\mu = \max_{i\neq j} \dfrac{\left|\langle A_i, A_j \rangle\right|}{\|A_i\|_2\cdot\|A_j\|_2},
\label{Def_mu}
\end{equation}
where $A_i$ denotes the $i$-th column of $A$ for $i=1,\ldots, n.$
\label{def_Coherence}
\end{definition}

\begin{definition}
For any matrix $A\in \mathbf{R}^{m\times n}$,
given $1\leq k\leq n,$ the restricted isometry constant (RIC) of $A$ with respect to $k$, $\delta_k$, is defined to be the smallest constant $\delta$ such that
\begin{equation}
(1-\delta) \|z\|_2^2 \leq \|Az\|_2^2 \leq (1+\delta) \|z\|_2^2,
\label{Def_RIP}
\end{equation}
for all $k$-sparse vector, i.e., $\|z\|_0 \leq k.$
\label{def_RIP}
\end{definition}

By Definition {\ref{def_RIP}}, RIC can also be written as:
\begin{equation}
\label{RIPEquivelent}
\delta_k = \sup_{z\in\mathbf{R}^n \setminus \{0\}, \|z\|_0\leq k}\frac{\left|z^T(A^TA-\mathbf{I}_n)z\right|}{\|z\|_2^2},
\end{equation}
which is very similar to the middle part of (\ref{prop_gRIP_Exp}).
In fact, Proposition \ref{Prop_gRIP_Coherence_RIP} shows that coherence and RIP are two special cases of gRIP.
\begin{prop}
For any column-normalized matrix $A\in \mathbf{R}^{m\times n}$, that is, $\|A_j\|_2 =1 $ for $j=1,\ldots, n$,
it holds
\begin{enumerate}
\item[(i)]
$\beta_{k,1,\infty} = \mu,$ for $2\leq k \leq n.$

\item[(ii)]
$\beta_{k,2,2} = \delta_{k} ,$ for $1 \leq k \leq n.$
\end{enumerate}
\label{Prop_gRIP_Coherence_RIP}
\end{prop}

The proof of this proposition is shown in Appendix B.

\subsection{Uniqueness Theorem Characterized via gRIP}

We first give a lemma to show the relation between two different norms for a $k$-sparse vector space.
\begin{lemma}
For any vector $x\in \mathbf{R}^n$ with $\|x\|_0=k\leq n$, and for any $1 \leq q \leq p \leq \infty$,
then
\begin{equation}
\|x\|_p \leq \|x\|_q \leq k^{\frac{1}{q}-\frac{1}{p}} \|x\|_p.
\label{Norm_Equiv_Thm}
\end{equation}
\label{Norm_Equiv_Lemma}
\end{lemma}
This lemma is trivial based on the well-known norm equivalence theorem so the proof is omitted.
Note that Lemma \ref{Norm_Equiv_Lemma} is equivalent to
\begin{equation}
\|x\|_p \leq k^{\max\{\frac{1}{p}-\frac{1}{q},0\}} \|x\|_q, \forall  p, q \in [1,\infty].
\label{Norm_Equiv_Extension}
\end{equation}
With Lemma \ref{Norm_Equiv_Lemma}, the following theorem shows that a $k$-sparse solution of the equation $Ax=b$
will be the unique sparsest solution
if $A$ satisfies a certain gRIP condition.
\begin{thm}
Let $x^*$ be a $k$-sparse solution of $Ax =b$.
If $A$ satisfies $(2k,p,q)$-gRIP with
\[
0<\beta_{2k,p,q}< (2k)^{\min\{\frac{1}{q}-\frac{1}{p},0\}},
\]
then $x^*$ is the unique sparsest solution.
\label{Unique_Thm}
\end{thm}

The proof of Theorem {\ref{Unique_Thm}} is given in Appendix C.
According to Proposition {\ref{Prop_gRIP_Coherence_RIP}} and Theorem {\ref{Unique_Thm}}, we can obtain the following uniqueness results characterized via coherence and RIP, respectively.

\begin{coro}
\label{Uniq_Co}
Let $x^*$ be a $k$-sparse solution of the equation $Ax =b$.
If $\mu$ satisfies
\[
0<\mu < \frac{1}{2k},
\]
then $x^*$ is the unique sparsest solution.
\end{coro}

It was shown in \cite{CaiCoherence} that when $\mu<\frac{1}{2k-1}$, the $k$-sparse solution should be unique.
In another perspective, it can be noted that the condition $\mu < \frac{1}{2k}$ is equivalent to $k<\frac{1}{2\mu}$
while $\mu<\frac{1}{2k-1}$ is equivalent to $k<\frac{1}{2\mu}+\frac{1}{2}.$
Since $k$ should be an integer, these two conditions are almost the same.

\begin{coro}
\label{Uniq_RIP}
Let $x^*$ be a $k$-sparse solution of the equation $Ax =b$.
If $\delta_{2k}$ satisfies
\[
0<\delta_{2k}<1,
\]
then $x^*$ is the unique sparsest solution.
\end{coro}

According to {\cite{Candes08RIP}},
the RIP condition obtained in Corollary {\ref{Uniq_RIP}}
is the same as the state-of-the-art result and more importantly,
is tight in the sense that once the condition is violated, then we can construct two different signals with the same sparsity.

\section{Convergence Analysis}

In this section, we will study the convergence of the AIT algorithm based on the introduced gRIP.

\subsection{Characterization via gRIP}

To describe the convergence of the AIT algorithm, we first define
$$
L_1 = 2^{p-1}(k^*)^{\max\{1-\frac{p}{q},0\}}+(2^{p-1}-(c_2)^p+1)k^*,
$$
$$
L_2 = 2^{p}(2k^*)^{\max\{1-\frac{p}{q},0\}} + 2^{p-1}(c_1)^pk^*,
$$
and
$$L = \min\{\sqrt[p]{L_1},\sqrt[p]{L_2}\},$$
where $p\in [1,\infty), q\in [1,\infty]$ and $c_1, c_2$ are the corresponding boundedness parameters.

\begin{thm}
\label{conv_gRIP}
Let $\{x^{(t)}\}$ be a sequence generated by the AIT algorithm.
Assume that $A$ satisfies $(3k^*+1,p,q)$-gRIP with the constant $\beta_{3k^*+1,p,q}<\frac{1}{L}$,
and let
\begin{enumerate}
\item[(i)]
$k=k^*;$

\item[(ii)]
$\underline{s} < s < \overline{s}$, where
$$
\underline{s} = \dfrac{(2k^*)^{\max\{\frac{1}{q}-\frac{1}{p},0\}}-\frac{1}{L}}{(2k^*)^{\max\{\frac{1}{q}-\frac{1}{p},0\}}-\beta_{3k^*+1,p,q}},
$$
and
$$
\overline{s} = \dfrac{(2k^*)^{\max\{\frac{1}{q}-\frac{1}{p},0\}}+\frac{1}{L}}{(2k^*)^{\max\{\frac{1}{q}-\frac{1}{p},0\}}+\beta_{3k^*+1,p,q}}.
$$
\end{enumerate}
Then
$$
\|x^{(t)}-x^*\|_p \leq (\rho_s)^t \|x^*-x^{(0)}\|_p + \frac{sL}{1-\rho_s}\|A^T \epsilon\|_q,
$$
where $\rho_s = \gamma_sL<1$ with
$$\gamma_s = |1-s|(2k^*)^{\max \{\frac{1}{q}-\frac{1}{p},0\}}+s\beta_{3k^*+1,p,q}.$$
Particularly, when $\epsilon=0$,
it holds
$$
\|x^{(t)}-x^*\|_p \leq (\rho_s)^t \|x^*-x^{(0)}\|_p.
$$
\end{thm}

The proof of this Theorem is presented in Appendix D.
Under the conditions of this theorem, we can verify that $0<\rho_s <1.$
We first note that $\beta_{3k^*+1,p,q}<\frac{1}{L} < 1 \leq (2k^*)^{\max\{\frac{1}{q}-\frac{1}{p},0\}},$
then it holds $\underline{s}< 1< \overline{s}$. The definition of $\gamma_s$ gives
$\gamma_s=$
\begin{eqnarray*}
\left\{
\begin{array}{l}
(1-s)(2k^*)^{\max\{\frac{1}{q}-\frac{1}{p},0\}}+s\beta_{3k^*+1,p,q}, \ if \ \underline{s}< s\leq 1 \\
(s-1)(2k^*)^{\max\{\frac{1}{q}-\frac{1}{p},0\}}+s\beta_{3k^*+1,p,q}, \ if \ 1< s < \overline{s}
\end{array}
\right..
\end{eqnarray*}
If $\underline{s}< s\leq 1$, it holds
$$
\gamma_s < (1-\underline{s})(2k^*)^{\max\{\frac{1}{q}-\frac{1}{p},0\}}+\underline{s}\beta_{3k^*+1,p,q} = \frac{1}{L}.
$$
Similarly, if $1< s\leq \overline{s}$
$$
\gamma_s < (\overline{s}-1)(2k^*)^{\max\{\frac{1}{q}-\frac{1}{p},0\}}+\overline{s} \beta_{3k^*+1,p,q} = \frac{1}{L}.
$$
Therefore, we have $\gamma_s<\frac{1}{L}$ and thus,
$\rho_s = \gamma_s L <1.$

Theorem {\ref{conv_gRIP}} demonstrates that in the noise free case,
the AIT algorithm converges to the original sparse signal at a linear rate, while in the noisy case,
it also converges at a linear rate until reaching an error bound.
Moreover, it can be noted that the constant $\rho_s$ depends on the step size $s$.
Since $\beta_{3k^*+1,p,q}<\frac{1}{L} < (2k^*)^{\max\{\frac{1}{q}-\frac{1}{p},0\}},$ $\rho_s$ reaches its minimum at $s=1$.
The trend of $\rho_s$ with respect to $s$ is shown in Fig. \ref{DiagramRho}.
The optimal convergence rate is obtained when $s=1$.
This observation is consistent with the conclusion drawn in \cite{kyrillidis2011recipes}.

\begin{figure}[!ht]
  \centering
  \includegraphics[width=0.2\textwidth]{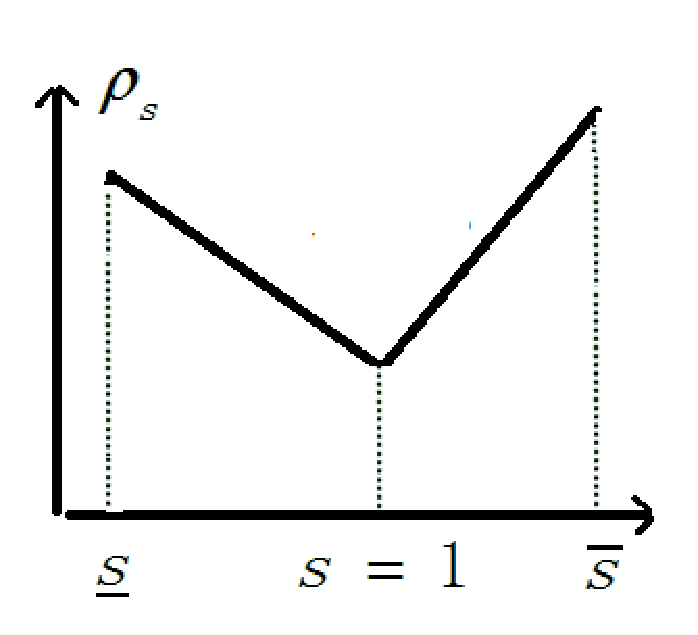}\\
  \caption{The trend of $\rho_s$ with respect to $s$.}
  \label{DiagramRho}
\end{figure}

By Proposition {\ref{Prop_gRIP_Coherence_RIP}}, it shows that the coherence and RIP are two special cases of gRIP,
thus we can easily obtain some recovery guarantees based on coherence and RIP respectively in the next two subsections.

\begin{remark}
From Theorem {\ref{conv_gRIP}}, we can see that the step size should lie in an appropriate interval,
which depends on the gRIP constant, which is generally NP-hard to verify.
However, we would like to emphasize that the theoretical result obtained in Thoeorem {\ref{conv_gRIP}} is of importance in theory
and it can give some insights and theoretical guarantees of the implementation of the AIT algorithm,
though it seems stringent.
Empirically, we find that a small interval of the step size, i.e., $[0.9,1]$ is generally sufficient for the convergence of the AIT algorithm.
This is also supported by the numerical experiments conducted in section VI.
In {\cite{kyrillidis2012combinatorial}}, it demonstrates that many algorithms perform well with either constant or adaptive step sizes.
In section VII C, we will discuss and compare different step-size schemes including the constant and an adaptive step-size strategies on the performance of AIT algorithms.
\end{remark}

\subsection{Characterization via Coherence}

Let $p=1, q = \infty.$ In this case,
$
L_1=(3-c_2)k^*
$,
$
L_2=(4+c_1)k^*,
$
and
$L=(3-c_2)k^*.$
According to Theorem {\ref{conv_gRIP}} and Proposition {\ref{Prop_gRIP_Coherence_RIP}},
assume that
$\mu <\frac{1}{(3-c_2)k^*}$, then the AIT algorithm converges linearly with the convergence rate constant
$$\rho_s = \gamma_s L = (|1-s|+s\mu)L<1$$
if we take $k=k^*$ and $\frac{1-\frac{1}{L}}{1-\mu}<s<\frac{1+\frac{1}{L}}{1+\mu}$.
In the following, we show that the constant $\gamma_s$ and thus $\rho_s$ can be further improved when $p=1$ and $q=\infty.$

\begin{thm}
\label{conv_Coherence}
Let $\{x^{(t)}\}$ be a sequence generated by the AIT algorithm for $b=Ax+\epsilon.$
Assume that $A$ satisfies $0<\mu< \frac{1}{(3-c_2)k^*}$,
and if we take
\begin{enumerate}
\item[(i)]
$k= k^*;$

\item[(ii)]
$1-\frac{1}{L} < s < \min\{\frac{1}{L\mu},1+\frac{1}{L}\},$
\end{enumerate}
then it holds
$$
\|x^{(t)}-x^*\|_1 \leq (\rho_s)^t \|x^*-x^{(0)}\|_1 + \frac{sL}{1-\rho_s}\|A^T \epsilon\|_{\infty},
$$
where  $\rho_s = \gamma_sL<1$ with
$$\gamma_s = \max\{|1-s|,s\mu\}.$$
Particularly, when $\epsilon=0$, it holds
$$
\|x^{(t)}-x^*\|_1 \leq (\rho_s)^t \|x^*-x^{(0)}\|_1.
$$
\end{thm}

The proof of this Theorem is given in Appendix E.
As shown in Theorem \ref{conv_Coherence}, the constant $\gamma_s$ can be improved from $|1-s|+s\mu$ to $\max\{|1-s|, s\mu\}$,
and also the feasible range of the step size parameter $s$ gets larger from $\left(\frac{1-\frac{1}{L}}{1-\mu},\frac{1+\frac{1}{L}}{1+\mu}\right)$ to $\left(1-\frac{1}{L},\min\{\frac{1}{L\mu},1+\frac{1}{L}\}\right).$
We list the coherence-based convergence conditions of several typical AIT algorithms in Table {\ref{TableCoherence}}.
As shown in Table {\ref{TableCoherence}}, it can be observed that the recovery condition for Soft algorithm is the same as those of OMP {\cite{TroppOMPCoherence}} and BP {\cite{DonohoCoherence}}.

\begin{table} [h]
\centering
\caption{Coherence based conditions for different AIT algorithms}
\label{TableCoherence}
\vspace*{10pt}
\begin{tabular}{ccccc}
\toprule
AIT     & Hard   & Half     & Soft  & SCAD   \\
\midrule
$c_2$            &  0            & 0     & 1             & 0   \\
$\mu$            &  $\frac{1}{3k^*-1}$   & $\frac{1}{3k^*-1}$    & $\frac{1}{2k^*-1}$         & $\frac{1}{3k^*-1}$   \\
\bottomrule
\end{tabular}
\end{table}

\subsection{Characterization via RIP}

Let $p=2, q = 2.$ In this case,
$L_1=2+(3-c_2^2)k^*, L_2=4+2c_1^2k^*,$ and thus
$$L=\min \{ \sqrt{4+2c_1^2k^*}, \sqrt{2+(3-c_2^2)k^*}\}.$$
According to Theorem \ref{conv_gRIP}, and by Proposition {\ref{Prop_gRIP_Coherence_RIP}}, we can directly claim the following corollary.

\begin{coro}
\label{conv_RIP}
Let $\{x^{(t)}\}$ be a sequence generated by the AIT algorithm for
$b=Ax+\epsilon.$
Assume that $A$ satisfies $\delta_{3k^*+1} < \frac{1}{L}$,
and if we take
\begin{enumerate}
\item[(i)]
$k=k^*;$

\item[(ii)]
$\underline{s} \leq s \leq \overline{s}$, where
$
\underline{s} = \dfrac{1-\frac{1}{L}}{1-\delta_{3k^*+1}}
$
and
$
\overline{s} = \dfrac{1+\frac{1}{L}}{1+\delta_{3k^*+1}}.
$
\end{enumerate}
Then
$$
\|x^{(t)}-x^*\|_2 \leq (\rho_s)^t \|x^*-x^{(0)}\|_2 + \frac{sL}{1-\rho_s}\|A^T \epsilon\|_2,
$$
where $\rho_s = \gamma_sL<1$ with $\gamma_s = |1-s|+s\delta_{3k^*+1}.$
Particularly, when $\epsilon=0$, it holds
$$
\|x^{(t)}-x^*\|_2 \leq (\rho_s)^t \|x^*-x^{(0)}\|_2.
$$
\end{coro}

According to Corollary {\ref{conv_RIP}},
the RIP based sufficient conditions for some typical AIT algorithms are listed in Table \ref{TableRIP}.

\begin{table} [h]
\centering
\caption{RIP based conditions for different AIT algorithms}
\label{TableRIP}
\vspace*{10pt}
\begin{tabular}{ccccc}
\toprule
AIT   &Hard   & Half     & Soft  & SCAD  \\
\midrule
$c_1$            &  0        & $1/3$          & 1             & 1  \\

$\delta_{3k^*+1}$  & $\frac{1}{2}$   & $\frac{3}{\sqrt{36+2k^*}}$  & $\frac{1}{\sqrt{2+2k^*}}$  & $\frac{1}{\sqrt{4+2k^*}}$\\
\bottomrule
\end{tabular}\\
\vspace{10pt}
\end{table}

Moreover, we note that the condition in Corollary {\ref{conv_RIP}} for Hard algorithm can be further improved via using the specific expression of the hard thresholding operator. This can be shown as the following theorem.

\begin{thm}
\label{Thm_RIP_Hard}
Let $\{x^{(t)}\}$ be a sequence generated by Hard algorithm for
$b=Ax+\epsilon.$
Assume that $A$ satisfies $\delta_{3k^*+1} < \frac{\sqrt{5}-1}{2}$,
and if we take $k=k^*$ and $s=1,$
then
\begin{align*}
\|x^{(t)}-x^*\|_2
&\leq \rho^t \|x^*-x^{(0)}\|_2 + \frac{\sqrt{5}+1}{2-2\rho}\|A^T \epsilon\|_2,
\end{align*}
where $\rho= \frac{\sqrt{5}+1}{2}\delta_{3k^*+1}<1$.
Particularly, when $\epsilon=0$, it holds
$$
\|x^{(t)}-x^*\|_2 \leq \rho^t  \|x^*-x^{(0)}\|_2.
$$
\end{thm}
The proof of Theorem \ref{Thm_RIP_Hard} is presented in Appendix F.

\section{Comparison with previous works}

This section discusses some related works of the AIT algorithm,
and then compares its computational complexity and sufficient conditions for convergence with other algorithms.

\subsubsection{On related works of the AIT algorithm}

In {\cite{MalekiITA2009}}, Maleki provided some similar results for two special AIT algorithms, i.e., Hard and Soft algorithms
with $k=k^*$ and $s=1$ for the noiseless case.
The sufficient conditions for convergence are $\mu<\frac{1}{3.1k^*}$ and  $\mu<\frac{1}{4.1k^*}$ for Hard and Soft algorithms, respectively.
In {\cite{ZengAIT2014}}, Zeng et al. improved and extended Maleki's results to a wide class of the AIT algorithm with step size $s=1$.
The sufficient condition based on coherence was improved to $\mu < \frac{1}{(3+c_1)k^*},$ where the boudedness parameter $c_1$ can be found in Table {\ref{TableBoundPar}}.
Compared with these two tightly related works, several significant improvements are made in this paper.
\begin{enumerate}
\item[(i)]
{\bf Weaker convergence conditions.}
The conditions obtained in this paper is weaker than those in both {\cite{MalekiITA2009}} and {\cite{ZengAIT2014}}.
More specifically, we give a unified convergence condition based on the introduced gRIP.
Particularly, as shown in Theorem {\ref{conv_Coherence}}, the coherence based conditions for convergence are $\mu<\frac{1}{(3-c_2)k^*-1}$,
which is much better than the condition $\mu < \frac{1}{(3+c_1)k^*}$ obtained in {\cite{ZengAIT2014}}.
Moreover, except Hard algorithm, we firstly show the convergence of the other AIT algorithms  based on RIP.

\item[(ii)]
{\bf Better convergence rate.}
The asymptotic linear convergence rate was justified in both {\cite{MalekiITA2009}} and {\cite{ZengAIT2014}}.
However, in this paper, we show the global linear convergence rate of the AIT algorithm,
which means it converges at a linear rate from the first iteration.

\item[(iii)]
{\bf More general model.}
In this paper, besides the noiseless model $b=Ax$, we also consider the performance of the AIT algorithm for the noisy model $b=Ax+\epsilon$,
which is very crucial since the noise is almost inevitable in practice.

\item[(iv)]
{\bf More general algorithmic framework.}
In both {\cite{MalekiITA2009}} and {\cite{ZengAIT2014}}, the AIT algorithm was only considered with unit step size ($s=1$).
While in this paper, we show that the AIT algorithm converges when $s$ is in an appropriate range.

\end{enumerate}

Among these AIT algorithms, Hard algorithm has been widely studied.
In {\cite{CaiCoherence}}, it was demonstrated that if $A$ has unit-norm columns and coherence $\mu$, then $A$ has the $(r,\delta_r)$-RIP with
\begin{equation}
\delta_r \leq (r-1)\mu.
\label{RIP-Coh}
\end{equation}
In terms of RIP, Blumensath and Davies {\cite{Blumensath08CS}} justified the performance of Hard algorithm when applied to signal recovery problem.
It was shown that if $A$ satisfies a certain RIP with $\delta_{3k^*}<\frac{1}{\sqrt{32}}$,
then Hard algorithm has global convergence guarantee.
Later, Foucart improved this condition to $\delta_{3k^*}<\frac{1}{2}$ or $\delta_{2k^*}<\frac{1}{4}$ {\cite{FoucartRIP2010}} and further improved it to $\delta_{3k^*}<\frac{1}{\sqrt{3}}\approx 0.5773$ (Theorem 6.18, \cite{Foucart2013Introduction}).
Now we can improve this condition to $\delta_{3k^*+1}<\frac{\sqrt{5}-1}{2} \approx 0.618$ as shown by Theorem {\ref{Thm_RIP_Hard}}.

\subsubsection{On comparison with other algorithms}

For better comparison, we list the state-of-the-art results on sufficient conditions of some typical algorithms including BP, OMP, CoSaMP, Hard, Soft, Half and general AIT algorithms in Table {\ref{TableERC}}.

\begin{table} [h!]
\caption{Sufficient Conditions for Different Algorithms}
\label{TableERC}
\begin{tabular}{ccc}
\toprule
Algorithm   & $\mu$                          & $(r, \delta_r)$ \\ \midrule
BP          & $\frac{1}{2k^*-1}^{({\cite{DonohoCoherence}})}$       & $(2k^*,0.707)^{(\cite{CaiSharpRIP2014})}$  \\ \hline
OMP         & $\frac{1}{2k^*-1}^{({\cite{TroppOMPCoherence}})}$       & $(13k^*,\frac{1}{6})^{(\text{Thm. 6.25},\cite{Foucart2013Introduction})}$  \\ \hline
CoSaMP      & $\frac{0.384}{4k^*-1}^\star$   & $(4k^*,0.384)^{({\cite{CoSaMP}})}$ \\ \hline
Hard        & $\frac{1}{3k^*-1}^{\text{(Thm. {\ref{conv_Coherence}})}}$ & $(3k^*$+1$,0.618)^{\text{(Thm. {\ref{Thm_RIP_Hard}})}}$\\ \hline
Soft        & $\frac{1}{2k^*-1}^{\text{(Thm. {\ref{conv_Coherence}})}}$                & $(3k^*$+1$,\frac{1}{\sqrt{2+2k^*}})^{\text{(Coro. {\ref{conv_RIP}})}}$ \\ \hline
Half        & $\frac{1}{3k^*-1}^{\text{(Thm. {\ref{conv_Coherence}})}}$              & $(3k^*$+1$,\frac{3}{\sqrt{36+2k^*}})^{\text{(Coro. {\ref{conv_RIP}})}}$ \\ \hline
General AIT   & $\frac{1}{(3-c_2)k^*-1}^{\text{(Thm. {\ref{conv_Coherence}})}}$         & $(3k^*$+1$,\frac{1}{\sqrt{4+2c_1^2k^*}})^{\text{(Coro. {\ref{conv_RIP}})}}$\\ \bottomrule
\end{tabular}
\\
$\star$: a coherence based sufficient condition for CoSaMP derived by the fact that $\delta_{4k^*}<0.384$ and $\delta_r\leq (r-1)\mu$.
\vspace*{-10pt}
\end{table}

From Table {\ref{TableERC}}, in the perspective of coherence,
 the sufficient conditions of AIT algorithms are slightly stricter than those of BP and OMP algorithms except Soft algorithm.
However, AIT algorithms are generally faster than both BP and OMP algorithms with lower computational complexities, especially for large scale applications due to their linear convergence rates.
As shown in the next section, the number of iterations required for the convergence of the AIT algorithm is empirically of the same order of the original sparsity level $k^*$, that is, $\mathcal{O}(k^*)$.
At each iteration of the AIT algorithm, only some simple matrix-vector multiplications and a projection on the vector need to be done,
and thus the computational complexity per iteration is $\mathcal{O}(mn)$.
Therefore, the total computational complexity of the AIT algorithm is $\mathcal{O}(k^*mn)$.
While the total computational complexities of BP and OMP algorithms are generally $\mathcal{O}(m^2n)$ and $\max\{\mathcal{O}(k^*mn), \mathcal{O}(\frac{(k^*)^2(k^*+1)^2}{4})\}$, respectively.
It should be pointed out that the computational complexity of OMP algorithm is related to the commonly used halting rule of OMP algorithm,
that is, the number of maximal iterations is set to be the true sparsity level $k^*$.

 Another important greedy algorithm, CoSaMP algorithm, identifies multicomponents (commonly $2k^*$) at each iteration.
From Table \ref{TableERC}, the RIP based sufficient condition of CoSaMP is $\delta_{4k^*}<0.384$ and a deduced coherence based sufficient condition is $\mu<\frac{0.384}{4k^*-1}$.
In the perspective of coherence, our conditions for AIT algorithms are better than CoSaMP, though this comparison is not very reasonable.
On the other hand, our conditions for AIT algorithms except Hard algorithm are generally worse than that of CoSaMP in the perspective of RIP.
However, when the true signal is very sparse, the conditions of AIT algorithms may be better than that of CoSaMP.
At each iteration of CoSaMP algorithm, some simple matrix-vector multiplications and a least squares problem should be considered.
Thus, the computational complexity per iteration of CoSaMP algorithm is generally $\max\{\mathcal{O}(mn), \mathcal{O}((3k^*)^3)\}$,
which is higher than those of AIT algorithms, especially when $k^*$ is relatively large.

Besides BP and greedy algorithms, another class of tightly related algorithms is the reweighted
techniques that have been also widely used for solutions to $\ell_{q}$
regularization with $q\in(0,1)$. Two well-known examples of such reweighted
techniques are the iterative reweighted least squares (IRLS) method {\cite{FOCUSS}} and
the reweighted $l_{1}$ minimization (IRL1) method {\cite{Candes2008RL1}}.
The convergence analysis conducted in {\cite{DaubechiesIRLS}} shows
that the IRLS method converges with an asymptotically superlinear convergence rate
under the assumptions that $A$ possesses a certain null-space property (NSP).
However, from Theorem {\ref{conv_gRIP}}, the rates of convergence of AIT algorithms are globally linear.
Furthermore, Lai et al. {\cite{IRLSLai2013}} applied the IRLS method to the unconstrained $l_q$ minimization problem
and also extended the corresponding convergence results to the matrix case.
It was shown also in {\cite{IRL1Chen}} that the IRL1 algorithm can converge to a stationary
point and the asymptotic convergence speed is approximately linear when applied to the unconstrained $l_q$ minimization problem.
Both in {\cite{IRLSLai2013}} and {\cite{IRL1Chen}},
the authors focused on the unconstrained $l_q$ minimization problem with a fixed regularization parameter $\lambda$,
while in this paper, we focus on a different model with an adaptive regularization parameter.

\section{Discussion}

In this section, we numerically discuss some practical issues on the implementation of AIT algorithms,
especially, the effects of several algorithmic factors including the estimated sparsity level parameter,
the column-normalization operation, different step-size strategies as well as the formats of different thresholding operators
on the performance of AIT algorithms.
Moreover,  we will further demonstrate the performance of several typical AIT algorihtms including Hard, Half and SCAD
via comparing with many state-of-the-art algorithms
such as CGIHT \cite{Blanchard2014}, CoSaMP {\cite{CoSaMP}}, 0-ALPS(4) \cite{kyrillidis2011recipes}
in the perspective of the 50\% phase transition curves \cite{Blanchard2013,Sturm2011}.

\subsection{Robustness of the estimated sparsity level}

In the preceding proposed algorithms,
the specified sparsity level parameter $k$ is taken
exactly as the true sparsity level $k^*$,
which is generally unknown in practice.
Instead, we can often obtain a rough estimate of the true sparsity level.
Therefore,
in this experiment,
we will explore the performance of the AIT algorithm with a variety of specified sparsity levels.
We varied $k$ from 1 to 150 while kept $k^*=15.$
The experiment setup is the same with Section VI. A.

\begin{figure}
\begin{minipage}[b]{.49\linewidth}
\centering
\includegraphics*[scale=0.49]{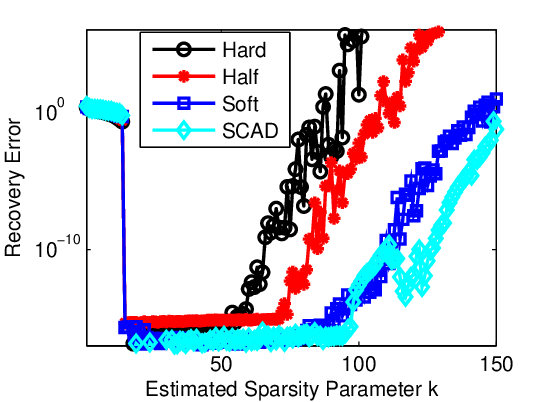}
\centerline{{\small (a) Robust (Noiseless)}}
\end{minipage}
\hfill
\begin{minipage}[b]{.49\linewidth}
\centering
\includegraphics*[scale=0.49]{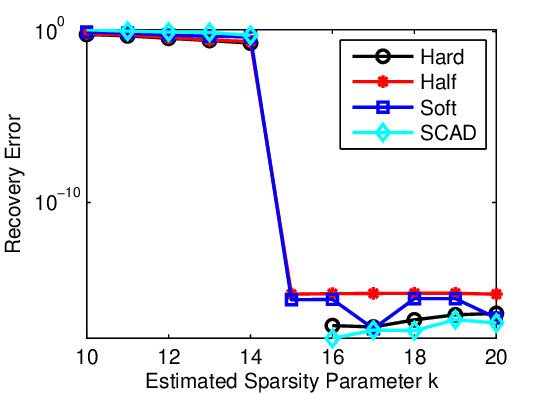}
\centerline{{\small (b) Detail (Noiseless)}}
\end{minipage}
\hfill
\begin{minipage}[b]{.49\linewidth}
\centering
\includegraphics*[scale=0.49]{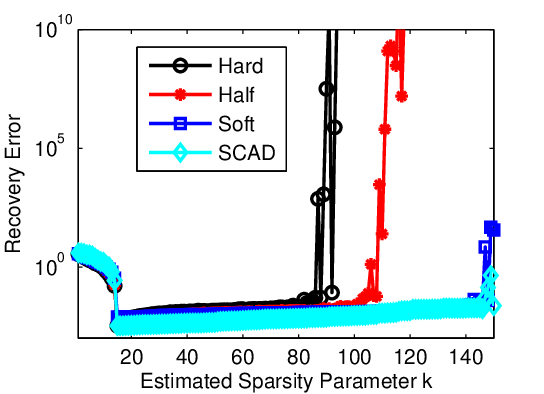}
\centerline{{\small (c) Robust (Noisy)}}
\end{minipage}
\hfill
\begin{minipage}[b]{.49\linewidth}
\centering
\includegraphics*[scale=0.49]{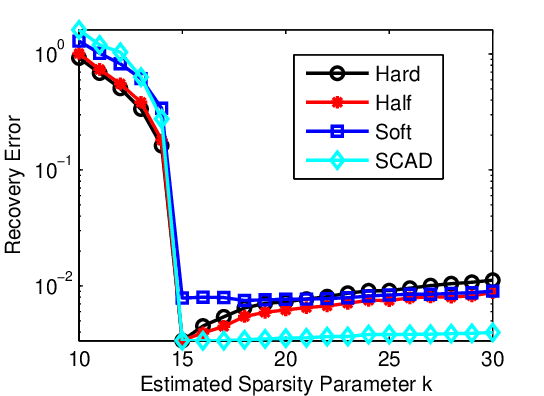}
\centerline{{\small (d) Detail (Noisy)}}
\end{minipage}
\hfill
\caption{
On robustness of the specified sparsity level.
(a) The trends of the recovery precision with different estimated sparsity levels in noiseless case.
(b) The detailed trends of the recovery precision with different estimated sparsity levels in noiseless case.
(c) The trends of the recovery precision with different estimated sparsity levels in noiseless case.
(d) The detailed trends of the recovery precision with different estimated sparsity levels in noisy case.
}
\label{Exp_Robust_k}
\end{figure}

From Fig. {\ref{Exp_Robust_k}}, we can observe that these AIT algorithms are efficient for a wide range of  $k$.
Interestingly, the point $k=k^*$ is a break point of the performance of all these AIT algorithms. When $k<k^*,$ all AIT algorithms fail to recover the original sparse signal,
while when $k\geq k^*$, a wide interval of $k$ is allowed for small recovery errors,
as shown in Fig. {\ref{Exp_Robust_k}} (b) and (d).
In the noise free case, if $\|x^{(t)}-x^*\|_2<10^{-10}$, the feasible intervals of the specified sparsity level $k$
are $[15,109]$ for SCAD and Soft, $[15,81]$ for Half and $[15,65]$ for Hard, respectively.
This observation is very important for real applications of AIT algorithms because $k^*$ is usually unknown.
In the noisy case, if $\|x^{(t)}-x^*\|_2<10^{-2}$, the feasible intervals of sparsity level $k$ are
$[15,105]$ for SCAD, $[15,40]$ for Soft, $[15,37]$ for Half and $[15,26]$ for Hard, respectively.

\subsection{With vs Without Normalization}

As shown in Algorithm 1, the column-normalization on the measurement matrix $A$ is required in consideration of
a clearer definition of the introduced gRIP and more importantly, better theoretical analyses.
However, in this subsection, we will conduct a series of simulations to show that such requirement is generally not necessary in practice.
The experiment setup is similar to Section VI.A.
More specifically, we set $m=250,$ $n=400$ and $k^* =15$.
The nonzero components of $x^*$ were generated randomly according to the standard Gaussian distribution.
The matrix $A$ was generated from i.i.d Gaussian distribution $\mathcal{N}(0,1/250)$ without normalization.
In order to adopt Algorithm 1, we let $\Lambda$ be the corresponding column-normalized factor matrix of $A$
(i.e., $\Lambda$ is a diagonal matrix and its diagonal element is the $l_2$-norm of the corresponding column of $A$),
and $\hat{A} = A\Lambda^{-1}$ be the corresponding column-normalized measurement matrix.
Assume that $\hat{x}$ is a recovery via Algorithm 1 corresponding to $\hat{A}$,
then $\bar{x} = \Lambda^{-1}\hat{x}$ is the corresponding recovery of $x^*$.
For each algorithm, we conducted 10 times experiments independently in both noiseless and noise (signal-to-noise ratio (SNR): 60dB) cases,
and recorded the average recovery precision.
The recovery precision is defined as $\frac{\|\tilde{x} -x^*\|_2}{\|x^*\|_2}$,
where $\tilde{x}$ and $x^*$ represent the recovery and original signal, respectively.
The experiment results are shown in Table {\ref{TableAITnoiseless}} and {\ref{TableAITnoise}}.

\begin{table}

\caption{The recovery precision of different AIT algorithms with or without column-normalization (noiseless case)}
  \begin{tabular}{ccccc}
  \toprule
  Algorithm& Hard&Soft&Half&SCAD\\
  \midrule
  no normalization&5.719e-6&1.425e-8&5.062e-6&9.330e-9\\
  normalization&5.703e-5&1.437e-8&5.935e-5&8.505e-9\\
  \bottomrule
  \end{tabular}
  \label{TableAITnoiseless}
\end{table}

\begin{table}

\caption{The recovery precision of different AIT algorithms with or without column-normalization (with 60dB noise)}
  \begin{tabular}{ccccc}
  \toprule
  Algorithm& Hard&Soft&Half&SCAD\\
  \midrule
  no normalization&1.217e-3&5.739e-3&1.206e-3&1.282e-3\\
  normalization&1.214e-3&5.498e-3&1.205e-3&1.264e-3\\
  \bottomrule
  \end{tabular}
  \label{TableAITnoise}
\end{table}

From Table {\ref{TableAITnoiseless}} and {\ref{TableAITnoise}},
we can see that the column-normalization operator has almost no effect on the performance of the AIT algorithm in both noiseless and noise cases.
Therefore, in the following experiments,
we will adopt the more practical AIT algorithm without the column-normalization for better comparison with the other algorithms.

\subsection{Constant vs Adaptive Step Size}

From Algorithm 1, we only consider the constant step-size.
However, according to many previous and empirical studies \cite{kyrillidis2011recipes, Blumensath2009},
we have known that certain adaptive step-size strategies may improve the performance of AIT algorithms.
In this subsection, we will compare the performance of two different step-size schemes, i.e.,
the constant step-size strategy and an adaptive step-size strategy introduced in {\cite{Blumensath2009}}
via the so-called 50\% phase transition curve \cite{Sturm2011}.
More specifically, the adaptive step-size scheme can be described as follows.
Assume that $x^{(t)}$ is the $t$-th iteration, then at $(t+1)$-th iteration, the step size $s^{(t+1)}$ is set as
\begin{equation}\label{AdaptiveSize}
s^{(t+1)} = \frac{\|(A^T(b-Ax^{(t)}))_{I^t}\|^2}{\|(AA^T(b-Ax^{(t)}))_{I^t}\|^2},
\end{equation}
where $I^t$ is the support set of $x^{(t)}$, $A$ is the measurement matrix and $b$ is the measurement vector.
Similar to {\cite{Blumensath2009}}, we will call the AIT algorithm with such adaptive step-size strategy the normalised AIT (NAIT) algorithm,
and correspondingly, several typical AIT algorithms such as Hard, Soft, Half and SCAD algorithms with such adaptive step-size strategy NHard,
NSoft, NHalf and NSCAD for short, respectively.
Note that NHard algorithm studied here is actually the same with the normalised iterative hard thresholding (NIHT) algorithm proposed in {\cite{Blumensath2009}}.

50\% phase transition curve was first introduced in \cite{Donoho06PhaseTransition} and
has been widely used to compare the recovery ability for different algorithms in compressed sensing \cite{Blanchard2013,Sturm2011}.
For a fixed $n$,  any given problem setting $(k,m,n)$ can depict a point in the space $(m/n,k/m) \in (0,1]^2$.
For any algorithm, its 50\% phase transition curve is actually a function $f$ on the $(k/m,m/n)$ space.
More specifically, if the point $(m/n,k/m)$ lies below the curve of the algorithm, i.e. $k/m<f(m/n)$,
then it means the algorithm could recover the sparse signal from the given $(k,m,n)$-problem with high probability,
otherwise the successful recovery probability is very low \cite{Donoho06PhaseTransition}.
Moreover, the 50\% phase transition curve usually depends on the prior distribution of $x^*$ as depicted in many researches \cite{BlumensathHard08,Blanchard2013,Sturm2011}.

In these experiments, we consider two common distributions of $x^*$,
the first one is the standard Gaussian distribution,
and the second one is a binary distribution, which takes $-1$ or $1$ with an equal probability.
For any given $(k,m,n)$,
the measurement matrix $A\in \mathbf{R}^{m\times n}$ is generated from the Gaussian distribution $\mathcal{N}(0,\frac{1}{m})$,
and the nonzero components of the original $k$-sparse signal $x^*$ are generated
independently and identically distribution (i.i.d.) according to the Gaussian or binary distributions.
For any experiment, we consider it as a successful recovery if
\[
\frac{\|\tilde{x}-x^*\|_{\infty}}{\|x^*\|_{\infty}}\leq 10^{-3},
\]
where $x^*$ is the original sparse signal and $\tilde{x}$ is the corresponding recovery signal.
We set $n=512$, $m=50,100,...,500$.
To determine $f(m/n)$, we exploit a bisection search scheme as the same as the experiment setting in \cite{Blanchard2013}.
We compare the 50\% phase transition curves of Hard, Soft, Half and SCAD algorithms
with their adaptive step-size versions, i.e., NHard, NSoft, NHalf, NSCAD in Fig. \ref{Exp_normalize}.

From Fig. \ref{Exp_normalize} (a) and (c), we can see that the performances of
all AIT algorithms except Soft algorithm adopting the adaptive step-size strategy (\ref{AdaptiveSize})
are significantly better than those of the corresponding AIT algorithms with a constant step size in the Guassian case.
In this case, NSCAD has the best performance, then NHalf and NHard,
while NSoft is the worst.
The performance of NSCAD is slightly better than those of NHalf and NHard, and much better than NSoft.
While for the binary case, as shown in Fig. \ref{Exp_normalize} (b) and (d),
NSCAD breaks down with the curve fluctuating around 0.1 while NHalf and NHard still perform well.
In the binary case, Soft as well as NSoft perform the worst.
In addition, we can see that the performances of Soft and NSoft are almost the same in all cases,
which means that such adaptive step-size strategy (\ref{AdaptiveSize}) may not bring the improvement on the performance of Soft algorithm.
Moreover, some interesting phenomena can also be observed in Fig. \ref{Exp_normalize},
that is, the performance of the AIT algorithm depends to some extent on the choice of the thresholding operator,
and for different prior distributions of the original sparse signal, the AIT algorithm may perform very different.
For these phenomena, we will study in the future work.

\begin{figure}
\begin{minipage}[b]{.49\linewidth}
\centering
\includegraphics*[scale=0.4]{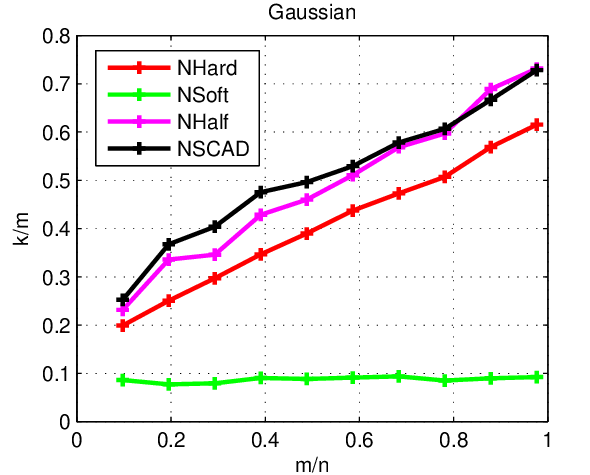}
\centerline{{\small (a) NAIT for Gaussian case}}
\end{minipage}
\hfill
\begin{minipage}[b]{.49\linewidth}
\centering
\includegraphics*[scale=0.4]{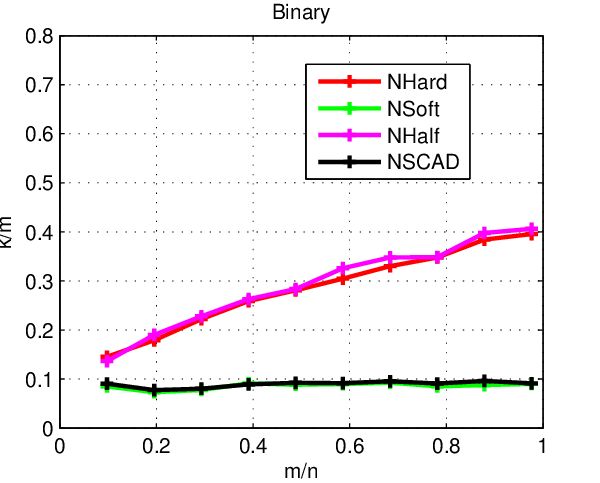}
\centerline{{\small (b) NAIT for Binary case}}
\end{minipage}
\hfill
\begin{minipage}[b]{.49\linewidth}
\centering
\includegraphics*[scale=0.4]{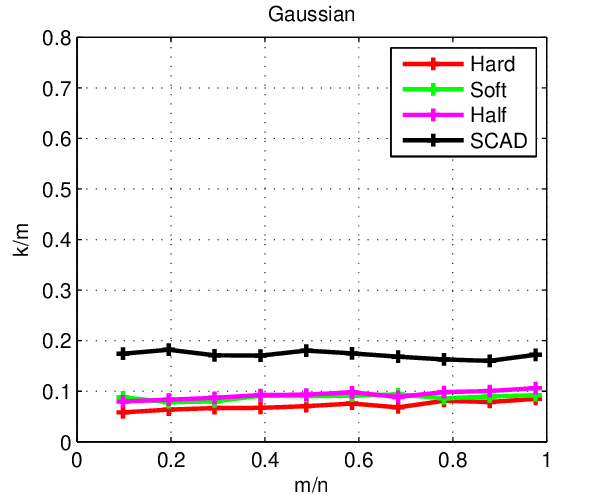}
\centerline{{\small (c) AIT for Gaussian case}}
\end{minipage}
\hfill
\begin{minipage}[b]{.49\linewidth}
\centering
\includegraphics*[scale=0.4]{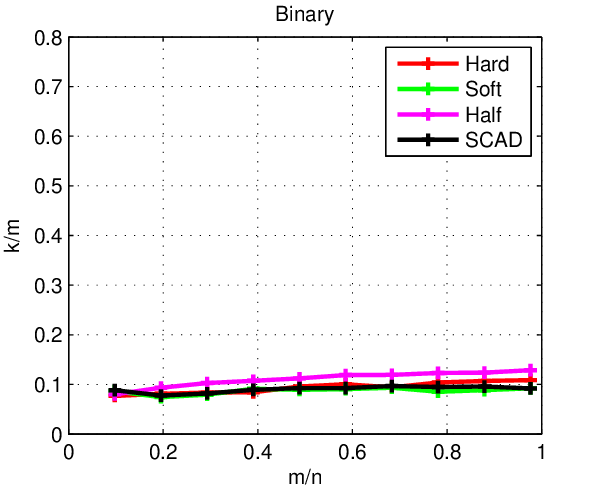}
\centerline{{\small (d) AIT for Binary case}}
\end{minipage}
\hfill
\caption{ 50\% phase transition curves of different AIT algorithms with two different step-size schemes.
(a) AIT algorithms with an adaptive step size for Gaussian case.
(a) AIT algorithms with an adaptive step size for Binary case.
(c) AIT algorithms with a constant step size for Gaussian case.
(d) AIT algorithms with a constant step size for Binary case.
}
\label{Exp_normalize}
\end{figure}

\subsection{Comparison with the State-of-the-art Algorithms}

We also compare the performance of several AIT algorithms including NHard, NSCAD and NHalf
with some typical state-of-the-art algorithms
such as conjugate gradient iterative hard thresholding (CGIHT) \cite{Blanchard2014}, CoSaMP {\cite{CoSaMP}}, 0-ALPS(4) \cite{kyrillidis2011recipes} in terms of their 50\% phase transition curves.
For more other algorithms like MP \cite{MallatMP1993}, HTP \cite{foucart2011hard}, OMP \cite{PatiOMP1993}, CSMPSP \cite{CSMPSP}, CompMP \cite{CompMP}, OLS \cite{OLS} etc.,
their 50\% phase transition curves can be found in \cite{Sturm2011}, and we omit them here.
For all considered algorithms, the estimated sparsity level parameters are set to be the true sparsity level of $x^*$.
The result is shown in Fig \ref{Exp_PhaseTransition}.

\begin{figure}
\begin{minipage}[b]{.49\linewidth}
\centering
\includegraphics*[scale=0.4]{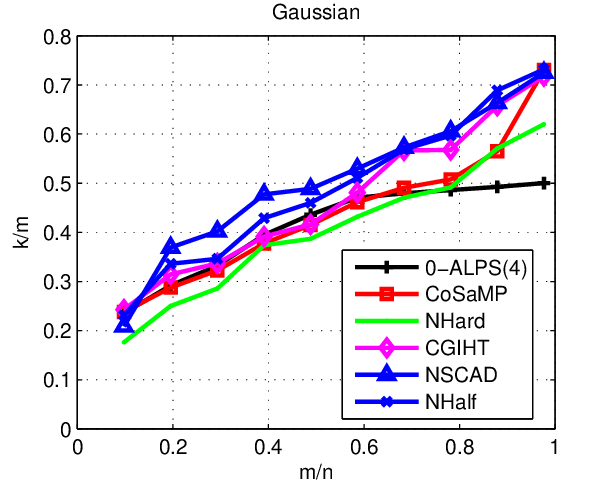}
\centerline{{\small (a) Gaussian Distribution}}
\end{minipage}
\hfill
\begin{minipage}[b]{.49\linewidth}
\centering
\includegraphics*[scale=0.4]{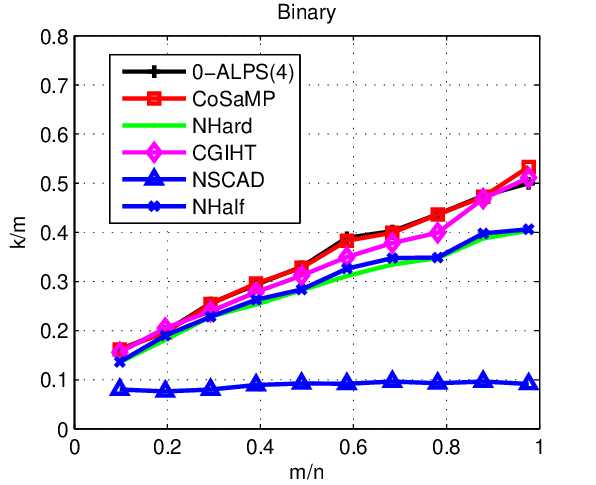}
\centerline{{\small (b) Binary Distribution}}
\end{minipage}
\hfill
\caption{ 50\% phase transition curves of different algorithms.
(a) Gaussian distribution case.
(b) Binary distribution case.
}
\label{Exp_PhaseTransition}
\end{figure}

From Fig. \ref{Exp_PhaseTransition},
we can see that almost all algorithms have better performances for the Gaussian distribution case than for the binary distribution case,
especially NSCAD algorithm.
More specifically,
as shown in Fig. \ref{Exp_PhaseTransition}(a), for the Gaussian distribution,
NSCAD has the best performance among all these algorithms,
and NHalf is slightly worse than NSCAD and better than the other algorithms.
While in the binary case, it can be seen from Fig. \ref{Exp_PhaseTransition}(b),
all AIT algorithms perform worse than the other algorithms like CGIHT, CoSaMP, 0-ALPS(4),
especially, NSCAD algorithm is much worse than the other algorithms.
These experiments demonstrate that AIT algorithms are more appropriate for the recovery problems
that the original sparse signals obey the Gaussian distribution.

\section{Conclusion}

We have conducted a study of a wide class of AIT algorithms for compressed sensing.
It should be pointed out that almost all of the existing iterative thresholding algorithms like Hard, Soft, Half and SCAD  are included in such class of algorithms.
The main contribution of this paper is the establishment of the convergence analyses of the AIT algorithm.
In summary, we have shown
when the measurement matrix satisfies a certain gRIP condition,
the AIT algorithm can converge to the original sparse signal at a linear rate in the noiseless case,
and approach to the original sparse signal at a linear rate until achieving an error bound in the noisy case.
As two special cases of gRIP, the coherence and RIP based conditions can be directly derived for the AIT algorithm.
Moreover, the tightness of our analyses can be demonstrated by two specific cases, that is,
the coherence-based condition for Soft algorithm
is the same as those of OMP and BP,
and the RIP based condition for Hard algorithm is
better than the recent result $\delta_{3k^*}<\frac{1}{\sqrt{3}}\approx 0.5773$ obtained in Theorem 6.18 in \cite{Foucart2013Introduction}.
Furthermore,
the efficiency of the algorithm and the correctness of the theoretical results are also verified via a series of numerical experiments.

In section VII, we have numerically discussed many practical issues on the implementation of AIT algorithms,
including the specified sparsity level parameter $k$, the column-normalization requirement as well as
different step-size setting schemes.
We can observe the following several interesting phenomena:
\begin{enumerate}
\item[(i)]
The AIT algorithm is robust to the specified sparsity level parameter $k$,
that is, the parameter $k$ can be specified in a large range to guarantee the well performance of the AIT algorithm.

\item[(ii)]
The column-normalization of the measurement matrix $A$ is not necessary for the use of AIT algorithms in the perspective of the recovery performance.

\item[(iii)]
Some adaptive step-size strategies may significantly improve the performance of AIT algorithms.

\item[(iv)]
The performance of AIT algorithm depends to some extent on the prior distribution of the original sparse signal.
Compared with the binary distribution, AIT algorithms are more appropriate for the recovery of the sparse signal generated by the Gaussian distribution.

\item[(v)]
The performance of the AIT algorithm depends on the specific thresholding operator.
\end{enumerate}
All of these phenomena are of interest, and we will study them in our future work.

\section*{Acknowledgements}

The authors would like to thank Prof. Wotao Yin,
Department of Mathematics, University of California, Los Angeles (UCLA), Los Angeles, CA, United States,
for a careful discussion of the manuscript, which led to a significant improvement of this paper.
Moreover, we thank the three anonymous reviewers, the associate editor and the potential reviewers for their
constructive and helpful comments.

This work was partially supported by the National 973 Programs (Grant No. 2013CB329404),
the Key Program of National Natural Science Foundation of China (Grants No.
11131006), the National Natural Science Foundations of China (Grants No. 11001227, 11171272, 11401462),
NSF Grants NSF DMS-1349855 and DMS-1317602.

\section*{Appendix A: Proof of Proposition {\ref{prop_gRIP1}}}

\begin{proof}
For any index set $S\subset \{1,\ldots, n\}$ with $|S|\leq k$ and a vector $x\in \mathbf{R}^{|S|},$
since $\frac{1}{p}+\frac{1}{q}=1$, then $\ell_p$ and $\ell_q$ norms are dual to each other,
which implies that
\begin{equation}
\label{lplq}
  \|(\mathbf{I}_{|S|}-A_S^TA_S)x\|_q=\sup_{y\in \mathbf{R}^{|S|}\setminus \{0\}}\dfrac{\left|y^T(\mathbf{I}_{|S|}-A_S^TA_S)x\right|}{\|y\|_p}.
\end{equation}
By Definition {\ref{def_gRIP}}, then
\begin{eqnarray}
  \label{def2}
  \beta_{k,p,q} = \sup_{|S|\leq k}\ \sup_{x,y\in \mathbf{R}^{|S|}\setminus \{0\}}\dfrac{\left|y^T(\mathbf{I}_{|S|}-A_S^TA_S)x\right|}{\|x\|_p\|y\|_p}.
\end{eqnarray}
It is obvious that
\begin{align*}
\beta_{k,p,q}
&\geq \sup_{|S|\leq k}\ \sup_{x\in \mathbf{R}^{|S|}\setminus \{0\}}\dfrac{\left|x^T(\mathbf{I}_{|S|}-A_S^TA_S)x\right|}{\|x\|_p^2}\\\nonumber
&=\sup_{z\in \mathbf{R}^n\setminus \{0\}, \|z\|_0\leq k}\frac{\left|z^T(\mathbf{I}_{n}-A^TA)z\right|}{\|z\|_p^2},
\end{align*}
which implies the right-hand side of (\ref{prop_gRIP_Exp}).

On the other hand, by (\ref{def2}), we can also observe that
\begin{align}
  \label{def3}
  \beta_{k,p,q} = \sup_{|S|\leq k}\ \sup_{\|x\|_p,\|y\|_p\leq 1} |y^Tx-y^TA_S^TA_Sx|,
\end{align}
and for any $x,y\in \mathbf{R}^{|S|},$
\begin{align}
\nonumber
&|y^Tx-y^TA_S^TA_Sx|\\\nonumber
&=\Big|\frac{1}{2}(\|x\|_2^2+\|y\|_2^2-\|x-y\|_2^2)\\\nonumber
&-\frac{1}{2}(\|A_Sx\|_2^2+\|A_Sy\|_2^2-\|A_Sx-A_Sy\|_2^2)\Big|\\\nonumber
&\leq \frac{1}{2}\Big|\|x\|_2^2-\|A_Sx\|_2^2\Big|+\frac{1}{2}\Big|\|y\|_2^2-\|A_Sy\|_2^2\Big|\\
&+\frac{1}{2}\Big|\|x-y\|_2^2-\|A_S(x-y)\|_2^2\Big|.
\label{upperBound}
\end{align}
Furthermore,
it can be noted that
\begin{align}
&\sup_{u,v\in \mathbf{R}^{|S|},\|u\|_p,\|v\|_p \leq 1}\Big|\|u-v\|_2^2-\|A_S(u-v)\|_2^2\Big|\nonumber\\
&\leq 4\sup_{w\in \mathbf{R}^{|S|}, \|w\|_p\leq 1}\Big|\|w\|_2^2-\|A_Sw\|_2^2\Big|,
\label{upperBound2}
\end{align}
since $\|u-v\|_p \leq 2$ for $\|u\|_p \leq 1$ and $\|v\|_p\leq 1.$
Plugging (\ref{upperBound}) and (\ref{upperBound2}) into (\ref{def3}), it yields
\begin{align*}
\beta^A_{k,p,q}
& \leq 3\sup_{|S|\leq k}\ \sup_{\|x\|_p\leq 1}\Big|\|x\|_2^2-\|A_Sx\|_2^2\Big|\\\nonumber
&=3\sup_{\|z\|_0\leq k,\|z\|_p\leq1}\Big|\|z\|_2^2-\|Az\|_2^2\Big|\\\nonumber
&=3\sup_{z\in\mathbf{R}^n\setminus \{0\},\|z\|_0\leq k}\frac{\left|z^T(A^TA-\mathbf{I}_n)z\right|}{\|z\|_p^2},
\end{align*}
which implies the left-hand side of (\ref{prop_gRIP_Exp}).
Therefore, the proof of this proposition is completed.
\end{proof}

\section*{Appendix B: Proof of Proposition {\ref{Prop_gRIP_Coherence_RIP}}}

\begin{proof}
(i) The definition of gRIP induces $\beta_{k,1,\infty}\geq \beta_{2,1,\infty}$ for all $k\geq 2$.
Therefore, if we can claim the following two facts: (a) $\beta_{2,1,\infty} \geq \mu$, and (b) $\beta_{k,1,\infty} \leq \mu$ for all $k\geq 2,$
then Proposition \ref{Prop_gRIP_Coherence_RIP} (i) follows.

We first justify the fact (a).
Suppose the maximal element of $\mathbf{I}_{n}-A^TA$ in magnitude appears at the $i_0$-th row and the $j_0$-th column. Because for any $j$, the $j$-th diagonal elements of $\mathbf{I}_{n}-A^TA$ equals to $1-\|A_j\|^2=0$, we know $i_0\neq j_0$.
Without loss of generality, we assume that $i_0 <j_0.$
Let $A_{i_0}$ and $A_{j_0}$ be the $i_0$-th and $j_0$-th column vector of $A$, respectively,
then Definition {\ref{def_Coherence}} gives
$$\mu = |A_{i_0}^TA_{j_0}|.$$
Let $S=\{i_0,j_0\}$ and $\mathbf{e}=(0,1)^T$.
Then
\begin{eqnarray}\nonumber
  \beta_{2,1,\infty}&\geq&\|(\mathbf{I}_{2}-A_S^TA_S)\mathbf{e}\|_\infty \\
  &=&\|\mathbf{e}-A_S^TA_{j_0}\|_\infty\nonumber\\
  &=& \mu.
  \label{gRIP-Coherence1}
\end{eqnarray}

Then we prove the fact (b).
For any vector $x\in \mathbf{R}^k$ and
a subset $S \subset \{1,2,\ldots, n\}$ with $|S|=k$,
let $B=\mathbf{I}_{k}-A_S^TA_S$ and $z=Bx$.
Then
$$
|z_i| = |\sum_{j=1}^k B_{ij}x_j| \leq\sum_{j=1}^k |B_{ij}x_j| \leq \mu \|x\|_1,
$$
for any $i=1,\dots, k.$
It implies that
$$
\|Bx\|_{\infty} \leq \mu \|x\|_1.
$$
By the definition of $\beta_{k,1,\infty},$ it implies
\begin{equation}
\beta_{k,1,\infty}\leq \mu.
\label{gRIP-Coherence2}
\end{equation}
According to (\ref{gRIP-Coherence1}) and (\ref{gRIP-Coherence2}), for all $2\leq k \leq n,$ it holds
$$
\beta_{k,1,\infty}= \mu.
$$

(ii)
From the inequality (\ref{prop_gRIP_Exp}) and the equality (\ref{RIPEquivelent}), we know
\begin{equation}
\delta_k\leq \beta_{k,p,q}.
\label{Prop2.3}
\end{equation}
To prove
\begin{equation}
\delta_k\geq \beta_{k,p,q},
\label{Prop2.4}
\end{equation}
note that equality (\ref{def3}) leads to
\begin{equation}
\beta_{k,p,q} \leq {\sup_{S\subset \{1,\ldots,n\},|S|\leq k}} \|\mathbf{I}_{|S|}-A_S^TA_S\|_{2},
\label{Prop2.5}
\end{equation}
and further
\begin{eqnarray}
\label{Prop2.6}
&&{\sup_{S\subset \{1,\ldots,n\},|S|\leq k}} \|\mathbf{I}_{|S|}-A_S^TA_S\|_{2} \\\nonumber
&&= {\sup_{|S|\leq k}}\ {\sup_{x\in\mathbf{R}^{|S|}\setminus \{0\}}}\frac{|x^T(\mathbf{I}_{|S|}-A_S^TA_S)x|}{\|x\|_2^2}\\\nonumber
&&= {\sup_{z\in\mathbf{R}^{n}\setminus \{0\},\|z\|_0\leq k}}\frac{|z^T(\mathbf{I}_{n}-A^TA)z|}{\|z\|_2^2}\\\nonumber
&&=\delta_k,
\end{eqnarray}
where the last equality holds by the equivalent definition of RIP (this can be also referred to Definition 1 in \cite{FoucartRIP2010}).
From (\ref{Prop2.3})-(\ref{Prop2.6}), we can conclude that
\[
\delta_k = \beta_{k,p,q}.
\]
\end{proof}

\section*{Appendix C: Proof of Theorem \ref{Unique_Thm}}

\begin{proof}
We prove this theorem by contradiction. Assume $x^{**}$ satisfies $Ax^{**}=b$ and $\|x^{**}\|_0 \leq k$.
Then
\[
A(x^{*}-x^{**}) = 0,
\]
which implies
\[
(\mathbf{I}_{n}-{A}^TA)(x^{*}-x^{**}) = x^{*}-x^{**}.
\]
Let $x=x^{*}-x^{**}$, $S$ be the support of $x$ and $x_S$ be a subvector of $x$ with the components restricted to $S$.
It follows
\[
(\mathbf{I}_{|S|}-{A_S}^TA_S)x_S = x_S,
\]
and further
\begin{equation}
\|(\mathbf{I}_{|S|}-{A_S}^TA_S)x_S\|_q = \|x_S\|_q,
\label{Lemma1.1}
\end{equation}
for any $q\in [1,\infty].$
Since $\|x^{*}\|_0 \leq k$ and $\|x^{**}\|_0 \leq k$,
then $|S| \leq 2k.$
For any $p\in [1,\infty),$ and by the definition of gRIP, we have
\[
\|(\mathbf{I}_{|S|}-{A_S}^TA_S)x_S\|_q \leq \beta_{2k,p,q} \|x_S\|_p.
\]
By Lemma {\ref{Norm_Equiv_Lemma}},
there holds
\[
\|x_S\|_p \leq (2k)^{\max\{\frac{1}{p}-\frac{1}{q},0\}}\|x_S\|_q.
\]
By the assumption of this theorem, then
\begin{align*}
\|(\mathbf{I}_{|S|}-{A_S}^TA_S)x_S\|_q
& \leq \beta_{2k,p,q} (2k)^{\max\{\frac{1}{p}-\frac{1}{q},0\}}\|x_S\|_q \nonumber\\
& < \|x_S\|_q,
\end{align*}
which contradicts with (\ref{Lemma1.1}).
Therefore, $x^{*}$ is the unique sparsest solution.
\end{proof}

\section*{Appendix D: Proof of Theorem \ref{conv_gRIP}}

Before justifying the convergence of the AIT algorithm based on gRIP,
we first introduce two lemmas.

\begin{lemma}
For any  $x, y \in \mathbf{R}^n$, and $p \in [1,\infty),$
then
\begin{equation}
\|x+y\|_p^p \leq 2^{p-1} (\|x\|_p^p+\|y\|_p^p).
\label{p_p_norm1}
\end{equation}
Moreover, if $x_i\cdot y_i \geq 0$ for $i=1,\ldots, n,$ then
\begin{equation}
\|x+y\|_p^p \geq \|x\|_p^p + \|y\|_p^p.
\label{p_p_norm2}
\end{equation}
\label{lp_norm}
\end{lemma}

The proof of Lemma {\ref{lp_norm}} is obvious since $f(z) = z^p$ is convex for $p\geq 1$ and any $z\geq 0$,
and $\|x\|_p \leq \|x\|_1$ for any $x\in \mathbf{R}^n$.
We will omit it due to the limitation of the length of the paper.

%
%
%

\begin{lemma}
For any $t\geq 1$ and $q\in [1,\infty]$, if $k\geq k^*$, the following inequality holds for the AIT algorithm:
\begin{equation}
\tau^{(t)}\leq (\sum_{i\in I^t_+}|z^{(t)}_i-x^*_i|^{q})^{1/q} = \|z^{(t)}_{I^t_+}-x^*_{I^t_+}\|_q,
\label{tau_bound}
\end{equation}
where $I^t_+$ is the index set of the largest $k+1$ components of $z^{(t)}$ in magnitude.
\label{tau}
\end{lemma}

\begin{proof}
When $q=\infty$, we need to show
\begin{equation}
\tau^{(t)}\leq \max_{i\in I^t_+} |z^{(t)}_i-x^*_i|,
\label{tau_infty}
\end{equation}
then Lemma \ref{Norm_Equiv_Lemma} shows that  (\ref{tau_bound}) holds for all $q \in [1,\infty]$.

Let $I^t$ be the index set of the largest $k$ components of $z^{(t)}$ in magnitude, then $I^t_+ = I^t \cup \{[k+1]\}$,
where  $[k+1]$ represents the index of the $(k+1)$-th largest component of $z^{(t)}$ in magnitude.
We will prove (\ref{tau_infty}) in the following two cases.

Case (i). If $I^*= I^t$, then
\begin{align}
\tau^{(t)}
&= |z_{[k+1]}^{(t)}| = |z_{[k+1]}^{(t)} - x_{[k+1]}^*| \nonumber\\
&\leq \max_{i\in I^t_+} |z^{(t)}_i-x^*_i|.
\label{Prop3_case1}
\end{align}

Case (ii). If $I^* \neq I^t$, then there exists $i_0 \in I^t$ such that
$$
i_0 \notin I^*.
$$
Otherwise $I^t\subset I^*$ and $I^t\neq I^*$ which contradicts with $|I^t|\geq k^*$ and $|I^*|=k^*$.
Thus, $x^*_{i_0} = 0$ and
\begin{align}
\tau^{(t)}
&= |z_{[k+1]}^{(t)}| \leq |z_{i_0}^{(t)}| = |z_{i_0}^{(t)} - x^*_{i_0}| \nonumber\\
&\leq \max_{i\in I^t_+} |z^{(t)}_i-x^*_i|.
\label{Prop3_case2}
\end{align}
Combining (\ref{Prop3_case1}) and (\ref{Prop3_case2}) gives (\ref{tau_infty}).
\end{proof}

\begin{proof}[Proof of Theorem \ref{conv_gRIP}]
In order to prove this theorem, we only need to justify the following two inequalities, i.e.,
for any $t\in \mathbf{N}$,
\begin{align}
\|z_{S^t}^{(t+1)}-x_{S^t}^*\|_q
\leq \gamma_s \|x^{(t)}-x^*\|_p + s\|A^T\epsilon\|_q,
\label{Thm4.2}
\end{align}
and for any $t\geq 1$,
\begin{align}
\|x^{(t)}-x^*\|_{p}
\leq L\|z_{S^{t-1}}^{(t)}-x^*_{S^{t-1}}\|_{q}.
\label{Thm4_Lf}
\end{align}
Then combining (\ref{Thm4.2}) and (\ref{Thm4_Lf}), it holds
\begin{align*}
&\|x^{(t+1)}-x^*\|_{p}
\leq L\|z_{S^{t}}^{(t+1)}-x^*_{S^{t}}\|_{q} \nonumber\\
&\leq \rho_s \|x^{(t)}-x^*\|_p + sL\|A^T\epsilon\|_q.
\end{align*}
Since $0<\rho_s<1$ under the assumption of this theorem, then by induction for any $t\geq 1,$ we have
\begin{equation*}
\|x^{(t)}-x^*\|_{p} \leq (\rho_s)^t\|x^*-x^{(0)}\|_p + \frac{sL}{1-\rho_s}\|A^T\epsilon\|_q.
\end{equation*}

First, we turn to prove the inequality (\ref{Thm4.2}).
By the \textbf{\emph{Step 1}} of Algorithm 1, for any $t\in \mathbf{N}$,
$$
z^{(t+1)}=x^{(t)}-sA^{T}(Ax^{(t)}-b),
$$
and we note that $b = Ax^*+\epsilon$,
then
\begin{align*}
&z^{(t+1)} - x^* = (\mathbf{I}_{n}-sA^{T}A)(x^{(t)}-x^*) + sA^T \epsilon\nonumber\\
& = (1-s)(x^{(t)}-x^*) + s(\mathbf{I}_{n}-A^{T}A)(x^{(t)}-x^*) + sA^T \epsilon.
\end{align*}
For any $t\in \mathbf{N}$ and $q\in [1,\infty],$ let $S^t=I^{t+1}_+\bigcup I^t\bigcup I^*$.
Noting that $I^t, I^*\subset S^t$, it follows
\[
A(x^{(t)}-x^*)=A_{S^t}(x^{(t)}_{S^t}-x^*_{S^t}).
\]
Then we have
\begin{eqnarray*}
&z^{(t+1)}_{S^t} - x^*_{S^t} = (1-s)(x^{(t)}_{S^t}-x^*_{S^t}) \nonumber\\
& + s(\mathbf{I}_{|S^t|}-A^{T}_{S^t}A_{S^t})(x^{(t)}_{S^t}-x^*_{S^t}) + sA^T_{S^t} \epsilon.
\end{eqnarray*}
Therefore,
\begin{align}
& \|z^{(t+1)}_{S^t}-x^*_{S^t}\|_q \leq |1-s| \cdot\|x^{(t)}_{S^t}-x^*_{S^t}\|_q \nonumber\\
&  + s \|(\mathbf{I}_{|S^t|}-A_{S^t}^{T}A_{S^t})(x_{S^t}^{(t)}-x_{S^t}^*)\|_q + s\|A_{S^t}^T\epsilon\|_q.
\label{diff_z(t+1)_x*}
\end{align}
Since $\|x^{(t)}\|_0 \leq k = k^*$ and $\|x^*\|_0 = k^*$
then
$$
|I^t| \leq k^*, |I^{t+1}_+|\leq k^*+1, |I^*| = k^*,
$$
and hence $|S^t| \leq 3k^*+1$.
For any $p\in [1,\infty)$, by (\ref{Norm_Equiv_Extension}) and the definition of gRIP (\ref{Def_gRIP}), it holds
\begin{equation}
\|x^{(t)}-x^*\|_q \leq (2k^*)^{\max\{\frac{1}{q}-\frac{1}{p},0\}}\|x^{(t)}-x^*\|_p,
\label{diff_x(t)_x*1}
\end{equation}
and
\begin{eqnarray}
&\|(\mathbf{I}_{|S^t|}-A_{S^t}^{T}A_{S^t})(x_{S^t}^{(t)}-x_{S^t}^*)\|_q \nonumber\\
&\leq \beta_{3k^*+1,p,q} \|x_{S^t}^{(t)}-x_{S^t}^*\|_p
= \beta_{3k^*+1,p,q} \|x^{(t)}-x^*\|_p.
\label{Thm4.1}
\end{eqnarray}
Plugging (\ref{diff_x(t)_x*1}) and (\ref{Thm4.1}) into (\ref{diff_z(t+1)_x*}), then
\begin{align*}
&\|z_{S^t}^{(t+1)}-x_{S^t}^*\|_q \nonumber\\
&\leq \left(|1-s|(2k^*)^{\max\{\frac{1}{q}-\frac{1}{p},0\}} + s \beta_{3k^*+1,p,q}\right)\|x^{(t)}-x^*\|_p \nonumber\\
& + s\|A_{S^t}^T\epsilon\|_q \nonumber\\
&\leq \gamma_s \|x^{(t)}-x^*\|_p + s\|A^T\epsilon\|_q.
\end{align*}
Thus, we have obtained the inequality ({\ref{Thm4.2}}).

Then we turn to the proof of (\ref{Thm4_Lf}).
We will prove it in two steps.

{\bf Step a):} For any $p\in [1,\infty)$,
\begin{equation}
\|x^{(t)}-x^*\|^p_{p}
=\|x^{(t)}_{I^*}-x^*_{I^*}\|^p_{p}+\|x^{(t)}_{I^t \setminus {I^*}}\|^p_{p}.
\label{Thm4.2.1}
\end{equation}
By Lemma \ref{lp_norm},
\begin{align}
&\|x^{(t)}_{I^*}-x^*_{I^*}\|^p_{p} = \|x^{(t)}_{I^*}-z^{(t)}_{I^*}+ z^{(t)}_{I^*}- x^*_{I^*}\|^p_{p} \nonumber\\
& \leq 2^{p-1}\|z^{(t)}_{I^*}-x^{(t)}_{I^*}\|^p_{p}+2^{p-1}\|z^{(t)}_{I^*}-x^*_{I^*}\|^p_{p}.
\label{Thm4.2.2}
\end{align}
Moreover, by the \textbf{\emph{Step 3}} of Algorithm 1 and Assumption 1,
for any $i\in I^t$:
$$
sgn(x^{(t)}_i) = sgn(z^{(t)}_i) \ \text{and}\ |x^{(t)}_i| \leq |z^{(t)}_i|.
$$
Thus, for any $i\in I^t \setminus I^*$, it holds
\begin{equation}
x^{(t)}_i \cdot (z^{(t)}_i - x^{(t)}_i) \geq 0.
\label{Thm4.2.4}
\end{equation}
With (\ref{Thm4.2.4}) and by Lemma \ref{lp_norm}, we have
\begin{eqnarray}
\|z^{(t)}_{I^{t} \setminus I^{*}}\|^p_{p} &=& \|x^{(t)}_{I^{t} \setminus I^{*}} + (z^{(t)}_{I^{t} \setminus I^{*}}-x^{(t)}_{I^{t} \setminus I^{*}})\|^p_{p}.\nonumber\\
&\geq& \|x^{(t)}_{I^t \setminus I^*}\|^p_{p} + \|z^{(t)}_{I^t \setminus I^*}-x^{(t)}_{I^t \setminus I^*}\|^p_{p}.
\label{Thm4.2.5}
\end{eqnarray}
Plugging (\ref{Thm4.2.2}) and (\ref{Thm4.2.5}) into (\ref{Thm4.2.1}), it becomes
\begin{align}
\|x^{(t)}-x^*\|^p_{p}
&\leq 2^{p-1}(\|z^{(t)}_{I^*}-x^{(t)}_{I^*}\|^p_{p}+\|z^{(t)}_{I^*}-x^*_{I^*}\|^p_{p}) \nonumber\\
&+\|z^{(t)}_{I^{t} \setminus I^*}\|^p_{p}-\|z^{(t)}_{I^{t} \setminus I^*}-x^{(t)}_{I^{t} \setminus I^*}\|^p_{p}.
\label{Thm4.3}
\end{align}
Furthermore, by the \textbf{\emph{Step 2}} of Algorithm 1, Assumption 1 and Lemma {\ref{tau}},
for any $t\geq 1$,
we have:
\begin{enumerate}
\item[(a)]
if $i\in I^t$, $c_2\tau^{(t)}\leq |z^{(t)}_i-x^{(t)}_i|\leq c_1\tau^{(t)} \leq \tau^{(t)}$;

\item[(b)]
if $i\not\in I^t$, $|z^{(t)}_i-x^{(t)}_i|=|z^{(t)}_{i}|\leq \tau^{(t)}$;

\item[(c)]
$\tau^{(t)} \leq \|z^{(t)}_{I^t_+}-x^*_{I^t_+}\|_{q}$.
\end{enumerate}
By the above facts (a)-(c), it holds
\begin{align}
\|z^{(t)}_{I^*}-x^{(t)}_{I^*}\|^p_{p}
&\leq k^* \max_{i\in I^*} |z_i^{(t)} - x_i^{(t)}|^p \leq k^* |\tau^{(t)}|^p,
\label{ztI-xtI}
\end{align}
and
\begin{align}
\|z^{(t)}_{I^t \setminus I^*}-x^{(t)}_{I^t \setminus I^*}\|^p_{p}
&\geq |I^t \setminus I^*| \min_{i\in I^t \setminus I^*} |z^{(t)}_i-x^{(t)}_i|^p \nonumber\\
&\geq  |I^t \setminus I^*|(c_2)^p|\tau^{(t)}|^p,
\label{II}
\end{align}
where $|I^t \setminus I^*|$ represents the cardinality of the index set $I^t \setminus I^*.$
Plugging (\ref{ztI-xtI}), (\ref{II}) into (\ref{Thm4.3}), it follows
\begin{align}
\|x^{(t)}-x^*\|^p_{p}
&\leq 2^{p-1}\|z^{(t)}_{I^*}-x^*_{I^*}\|^p_{p}+\|z^{(t)}_{I^t \setminus I^*}\|^p_{p}\nonumber\\
&+(2^{p-1}k^*-(c_2)^p|I^t \setminus I^*|)|\tau^{(t)}|^p.
\label{Thm4.4}
\end{align}
Furthermore, we note that
\begin{align*}
&\|z^{(t)}_{I^t \setminus I^*}\|_{p}^p = \|z^{(t)}_{I^t \setminus I^*} - x^*_{I^t \setminus I^*}\|_{p}^p \\
&\leq |I^t \setminus I^*| \max_{i\in I^t \setminus I^*} |z_i^{(t)} - x_i^*|^p \\
&= |I^t \setminus I^*| \cdot \|z^{(t)}_{I^t \setminus I^*}-x^*_{I^t \setminus I^*}\|^p_\infty\\
&\leq |I^t \setminus I^*| \cdot \|z^{(t)}_{I^t \setminus I^*}-x^*_{I^t \setminus I^*}\|^p_q,
\end{align*}
where the first equality holds because $x^*_{I^t \setminus I^*} =0$,
and the second inequality holds because of Lemma \ref{Norm_Equiv_Lemma}.
Therefore, (\ref{Thm4.4}) becomes
\begin{align}
&\|x^{(t)}-x^*\|^p_{p}
\nonumber\\
&\leq 2^{p-1}\|z^{(t)}_{I^*}-x^*_{I^*}\|^p_{p}+|I^t \setminus I^*| \cdot \|z^{(t)}_{I^t \setminus I^*}-x^*_{I^t \setminus I^*}\|^p_q\nonumber\\
& +\left(2^{p-1}k^*-(c_2)^p|I^t \setminus I^*|\right)|\tau^{(t)}|^p\nonumber\\
&\leq 2^{p-1}\|z^{(t)}_{I^*}-x^*_{I^*}\|^p_{p} \nonumber\\
&+(2^{p-1}k^*-(c_2)^p|I^t \setminus I^*|+|I^t \setminus I^*|) \|z^{(t)}_{I^t_+}-x^*_{I^t_+}\|^p_q\nonumber\\
&\leq 2^{p-1}(k^*)^{\max\{1-\frac{p}{q},0\}} \|z^{(t)}_{I^{*}}-x^*_{I^{*}}\|^p_{q} \nonumber\\
&+(2^{p-1}-(c_2)^p+1)k^*  \|z^{(t)}_{I^{t}_+}-x^*_{I^{t}_+}\|^p_q\nonumber\\
& \leq L_1 \|z^{(t)}_{S^{t-1}}-x^*_{S^{t-1}}\|^p_q,
\label{Thm4_L1.1}
\end{align}
where the second inequality holds by the fact (c), i.e.,
$\tau^{(t)} \leq \|z^{(t)}_{I^t_+}-x^*_{I^t_+}\|_q$, the third inequality holds by Lemma \ref{Norm_Equiv_Lemma} and  $|I^t \setminus I^*|\leq k^*$ and the last inequality holds because $S^{t-1}=I^t_+\cup I^{t-1}\cup I^*$.
Thus, it implies
\begin{equation}
\|x^{(t)}-x^*\|_{p}\leq \sqrt[p]{L_1}\|z^{(t)}_{S^{t-1}}-x^*_{S^{t-1}}\|_{q}.
\label{Thm4_L1}
\end{equation}

{\bf Step b):} By Lemma {\ref{lp_norm}},
\begin{align}
&\|x^{(t)}-x^*\|^p_p=\|x^{(t)}_{I^t}-x^*_{I^t}\|^p_{p}+\|x^*_{I^* \setminus I^t}\|^p_{p} \nonumber\\
&\leq 2^{p-1}\|z^{(t)}_{I^t}-x^*_{I^t}\|^p_{p}+2^{p-1}\|z^{(t)}_{I^t}-x^{(t)}_{I^t}\|^p_{p} \nonumber\\
&+2^{p-1}\|z^{(t)}_{I^* \setminus I^t}-x^*_{I^* \setminus I^t}\|^p_{p}+2^{p-1} \|z^{(t)}_{I^* \setminus I^t}\|_p^p \nonumber\\
& = 2^{p-1}\|z^{(t)}_{I^t \cup I^*}-x^*_{I^t \cup I^*}\|^p_{p}\nonumber\\
&+ 2^{p-1}(\|z^{(t)}_{I^t}-x^{(t)}_{I^t}\|^p_{p} + \|z^{(t)}_{I^* \setminus I^t}\|_p^p).
\label{xt-x*_p-norm}
\end{align}
Moreover, by Lemma {\ref{Norm_Equiv_Lemma}}, it holds
\begin{align}
&\|z^{(t)}_{I^t \cup I^*}-x^*_{I^t \cup I^*}\|^p_{p} \nonumber\\
&\leq (|I^t \cup I^*|)^{\max\{1-\frac{p}{q},0\}} \|z^{(t)}_{I^t \cup I^*}-x^*_{I^t \cup I^*}\|_q^p \nonumber\\
& \leq (2k^*)^{\max \{1-\frac{p}{q},0\}} \|z^{(t)}_{I^t \cup I^*}-x^*_{I^t \cup I^*}\|_q^p,
\label{zt-x*}
\end{align}
where the last inequality holds for $|I^t \cup I^*| \leq 2k^*$.
We also have
\begin{align}
&\|z^{(t)}_{I^t}- x^{(t)}_{I^t}\|^p_{p} \leq k^* \max_{i\in I^t} |z_i^{(t)} - x_i^{(t)}|^p \nonumber\\
&\leq k^* (c_1\tau^{(t)})^p \leq k^* (c_1)^p \|z^{(t)}_{I^t_+}-x^*_{I^t_+}\|_q^p.
\label{zt_I-xt_I_p-norm}
\end{align}
Since $|I^t| = |I^*| = k^*$, then
$$
|I^* \setminus I^t| = |I^t \setminus I^*|.
$$
Thus, it holds
\begin{align}
\|z^{(t)}_{I^* \setminus I^t}\|_p^p
&\leq |I^* \setminus I^t| \max_{i\in I^* \setminus I^t} |z_i^{(t)}|^p \leq |I^* \setminus I^t| \cdot |\tau^{(t)}|^p\nonumber\\
& = |I^t \setminus I^*|\cdot |\tau^{(t)}|^p \leq |I^t \setminus I^*| \min_{i\in I^t \setminus I^*} |z_i^{(t)}|^p \nonumber\\
&\leq \|z^{(t)}_{I^t \setminus I^*}\|_p^p = \|z^{(t)}_{I^t \setminus I^*}-x^*_{I^t \setminus I^*}\|_p^p \nonumber\\
& \leq (k^*)^{\max\{1-\frac{p}{q},0\}} \|z^{(t)}_{I^t \setminus I^*}-x^*_{I^t \setminus I^*}\|_q^p.
\label{zt_I*-It}
\end{align}
Plugging (\ref{zt-x*}), (\ref{zt_I-xt_I_p-norm}) and (\ref{zt_I*-It}) into (\ref{xt-x*_p-norm}),
and further since $S^{t-1}=I^t_+\cup I^{t-1}\cup I^*,$ and thus
$I^t_+ \subset S^{t-1}, I^t \subset I^t_+ \subset S^{t-1}, I^t\cup I^*\subset S^{t-1}, I^t \setminus I^*\subset S^{t-1},$
it becomes
\begin{align}
&\|x^{(t)}-x^*\|^p_p \nonumber\\
&\leq (2^{p}(2k^*)^{\max\{1-\frac{p}{q},0\}} + 2^{p-1}(c_1)^pk^*) \|z^{(t)}_{S^{t-1}}-x^*_{S^{t-1}}\|_q^p \nonumber\\
&= L_2 \|z^{(t)}_{S^{t-1}}-x^*_{S^{t-1}}\|_q^p.
\label{Otherhand}
\end{align}
Thus, we have
\begin{equation}
\|x^{(t)}-x^*\|_{p}\leq \sqrt[p]{L_2}\|z^{(t)}-x^*\|_{q}.
\label{Thm4_L2}
\end{equation}

From (\ref{Thm4_L1}) and (\ref{Thm4_L2}), for any $t\geq 1$, it holds
\begin{align}
\|x^{(t)}-x^*\|_{p}
&\leq \min \{\sqrt[p]{L_1},\sqrt[p]{L_2}\}\|z^{(t)}-x^*\|_{q} \nonumber\\
&= L\|z^{(t)}-x^*\|_{q}\nonumber\\
&= L\|z_{S^{t-1}}^{(t)}-x^*_{S^{t-1}}\|_{q},
\label{Thm4_L}
\end{align}
where the last equality holds for $S^{t-1}=I^t_+\cup I^{t-1}\cup I^*$.
Thus, we have obtained (\ref{Thm4_Lf}).

Therefore, we end the proof of this theorem.
\end{proof}

\section*{Appendix E: Proof of Theorem \ref{conv_Coherence}}

\begin{proof}
The proof is similar to that of Theorem {\ref{conv_gRIP}}.
According to the proof of Theorem {\ref{conv_gRIP}},
we have known that (\ref{diff_z(t+1)_x*})-(\ref{Thm4.1}) hold for all pairs of $(p,q)$ with $\frac{1}{p}+\frac{1}{q}=1,$
and thus obviously hold for $p=1$ and $q=\infty.$
In the following, instead of the inequality (\ref{Thm4.2}), we will derive a tighter upper bound of $\|z_{S^t}^{(t+1)}-x_{S^t}^*\|_{\infty},$ that is,
\begin{align}
&\|z_{S^t}^{(t+1)}-x_{S^t}^*\|_{\infty} \nonumber\\
&\leq
\max\{\mu s, |1-s|\}\|x^{(t)}-x^*\|_1
+s\|A^T\epsilon\|_\infty.
\label{Tigher_UpperBound}
\end{align}

Now we turn to prove the inequality (\ref{Tigher_UpperBound}).
According to (\ref{AIT1}), it can be observed that
\begin{align*}
&\|z^{(t+1)}-x^*\|_{\infty}\leq \\
&\|\left((1-s)\mathbf{I}_{n}+s(\mathbf{I}_{n}-A^TA)\right)(x^{(t)}-x^*)\|_\infty + s\|A^T\epsilon\|_\infty.
\end{align*}
Let $B = (1-s)\mathbf{I}_{n}+s(\mathbf{I}_{n}-A^TA)$ and $B_{ij}$ be the $(i, j)$-th element of $B$.
Since $\|A_j\|_2 = 1$ for all $j=1,\ldots, n,$ then
$$B_{ii} = 1-s,$$
for all $i=1,\ldots, n$.
Moreover, by the definition of the coherence $\mu$, the absolutes of all the off-diagonal elements of $\mathbf{I}_{n}-A^TA$ are no bigger than $\mu.$
Thus,
$$|B_{ij}| \leq s\mu,$$
for any $i\neq j.$
As a consequence, it holds
$$
\max_{i,j\in \{1,\dots, n\}} |B_{ij}| \leq \max \{|1-s|,s\mu\} = \gamma_s.
$$
Furthermore, for any $i=1,\ldots, n,$
\begin{align*}\nonumber
\left|z^{(t+1)}_i-x^*_i\right|
& = \left|\sum_{j=1}^n B_{ij}(x^{(t)}_j-x^*_j) + s A_i^T \epsilon\right|\\
&\leq \left|\sum_{j=1}^n B_{ij}(x^{(t)}_j-x^*_j)\right|+s\|A^T\epsilon\|_\infty\\
&\leq \gamma_s\|x^{(t)}-x^*\|_1+s\|A^T\epsilon\|_\infty .\nonumber
\end{align*}
This implies
\begin{align*}
\|z_{S^t}^{(t+1)}-x_{S^t}^*\|_{\infty}
\leq \gamma_s\|x^{(t)}-x^*\|_1+s\|A^T\epsilon\|_\infty .
\end{align*}
Therefore, we obtain the (\ref{Tigher_UpperBound}).
According to the proof of Theorem \ref{conv_gRIP}, we have that the inequality (\ref{Thm4_Lf}) still holds when $p=1$ and $q=\infty,$ that is,
\begin{align}
\|x^{(t)}-x^*\|_{1} \leq L\|z_{S^{t-1}}^{(t)}-x^*_{S^{t-1}}\|_{\infty}.
\label{coherence1}
\end{align}
Similar to the rest of the proof of Theorem {\ref{conv_gRIP}},
combining (\ref{Tigher_UpperBound}) and (\ref{coherence1}),
we can conclude the proof of this theorem.
\end{proof}

\section*{Appendix F: Proof of Theorem \ref{Thm_RIP_Hard}}

\begin{proof}
The proof of this theorem is also very similar to that of Theorem {\ref{conv_gRIP}}.
According to the proof of Theorem {\ref{conv_gRIP}},
we have known that (\ref{Thm4.2}) holds for all pairs of $(p,q)$ with $\frac{1}{p}+\frac{1}{q}=1,$
and thus obviously holds for $p=2$ and $q=2,$ that is,
\begin{align}
\|z_{S^t}^{(t+1)}-x_{S^t}^*\|_2
\leq \delta_{3k^*+1}\|x^{(t)}-x^*\|_2 + \|A^T\epsilon\|_2,
\label{Tigher_UpperBound_RIP}
\end{align}
where $S^{t} = I^{t+1}_+\cup I^{t}\cup I^*$, $I^{t+1}_+$ is the index set of the largest $k+1$ components of $z^{(t+1)},$
$I^t$ and $I^*$ represent the support sets of $x^{(t)}$ and $x^*,$ respectively.
In the following, instead of the inequality (\ref{Thm4_Lf}), we will derive a tighter upper bound of $\|x^{(t)}-x^*\|_2,$ that is,
\begin{align}
\|x^{(t)}-x^*\|_2\leq \frac{\sqrt{5}+1}{2}\|z_{S^{t-1}}^{(t)}-x_{S^{t-1}}^*\|_{2}.
\label{UpperBound_Diff_RIP}
\end{align}

Now we turn to prove the inequality (\ref{UpperBound_Diff_RIP}).
It can be noted that
\begin{equation}
\|x^{(t)}-x^*\|^2_2
=\|x^{(t)}_{I^{t}}-x^*_{I^t}\|^2_{2}+\|x^{(t)}_{I^* \setminus I^t}-x^*_{I^* \setminus I^t}\|^2_{2}.
\label{hardEq1}
\end{equation}
On one hand, since $x^{(t)}_i=z^{(t)}_i$ for any $i\in I^t$,
then
\begin{equation}
\|x^{(t)}_{I^t}-x^*_{I^t}\|^2_{2}= \|z^{(t)}_{I^t}-x^*_{I^t}\|^2_{2}.
\label{hardEq1.1}
\end{equation}
On the other hand, we can also observe that $x^{(t)}_i=0$ for any $i\in I^* \setminus I^t$,
and thus
\begin{align}
\nonumber
&\|x^{(t)}_{I^*\setminus I^t}-x^*_{I^* \setminus I^t}\|^2_{2}=\|x^*_{I^* \setminus I^t}\|^2_{2} =\sum_{i\in I^* \setminus I^t}(x^*_i-z^{(t)}_i+z^{(t)}_i)^2\nonumber\\
&\leq \sum_{i\in I^* \setminus I^t} \left(\frac{\sqrt{5}+3}{2}(x^*_i-z^{(t)}_i)^2+\frac{\sqrt{5}+1}{2}(z^{(t)}_i)^2\right)\nonumber\\
&\leq \sum_{i\in I^* \setminus I^t}\frac{\sqrt{5}+3}{2}(x^*_i-z^{(t)}_i)^2+\sum_{i\in I^t \setminus I^*}\frac{\sqrt{5}+1}{2}(z^{(t)}_i)^2\nonumber\\
&=\frac{\sqrt{5}+3}{2}\|z^{(t)}_{I^* \setminus I^t}-x^*_{I^* \setminus I^t}\|^2_{2}+\frac{\sqrt{5}+1}{2}\|z^{(t)}_{I^t \setminus I^*}\|^2_{2} \nonumber\\
&=\frac{\sqrt{5}+3}{2}\|z^{(t)}_{I^* \setminus I^t}-x^*_{I^* \setminus I^t}\|^2_{2}+\frac{\sqrt{5}+1}{2}\|z^{(t)}_{I^t \setminus I^*}-x^*_{I^t \setminus I^*}\|^2_{2}.
\label{hardEq2}
\end{align}
The first inequality holds by the following relation
$$
(a+b)^2=a^2+b^2+2ab\leq (1+\frac{\sqrt{5}+1}{2})a^2+(1+\frac{\sqrt{5}-1}{2})b^2
$$
for any $a, b\in \mathbf{R}$.
The second inequality holds due to the following facts:
\begin{enumerate}
\item[(a)]
for any $i\in {I^*\setminus I^t}, |z^t_i|\leq \tau^{(t)},$

\item[(b)]
for any $i\in {I^t \setminus I^*}, |z^t_i|\geq \tau^{(t)},$

\item[(c)]
$|I^*\setminus I^t|=|I^t \setminus I^*|$,
\end{enumerate}
and hence
$$\max_{i\in I^*\setminus I^t}|z^{(t)}_i|\leq \min_{i\in I^t \setminus I^*}|z^{(t)}_i|.$$
The last equality holds for $x^*_i=0, \forall i\in {I^t \setminus I^*}$.
Plugging (\ref{hardEq1.1}) and (\ref{hardEq2}) into (\ref{hardEq1}), we have
\begin{align*}
&\|x^{(t)}-x^*\|^2_2
\leq  \|z^{(t)}_{I^t}-x^*_{I^t}\|^2_{2}+\frac{\sqrt{5}+1}{2} \|z^{(t)}_{I^t \setminus I^*}-x^*_{I^t \setminus I^*}\|^2_{2}
\\\nonumber
&+\frac{\sqrt{5}+3}{2} \|z^{(t)}_{I^* \setminus I^t}-x^*_{I^* \setminus I^t}\|^2_{2}\\\nonumber
&=\|z^{(t)}_{I^t\bigcap I^*}-x^*_{I^t\bigcap I^*}\|^2_{2}+\frac{\sqrt{5}+3}{2} \|z^{(t)}_{I^t \setminus  I^*}-x^*_{I^t \setminus  I^*}\|^2_{2}
\\\nonumber
&+\frac{\sqrt{5}+3}{2}\|z^{(t)}_{I^* \setminus I^t}-x^*_{I^* \setminus I^t}\|^2_{2}\\\nonumber
&\leq \frac{\sqrt{5}+3}{2}\|z_{S^{t-1}}^{(t)}-x_{S^{t-1}}^*\|^2_{2},
\label{hard_Eq3}
\end{align*}
where $S^{t-1} = I^{t}_+\cup I^{t-1}\cup I^*$.
The last inequality holds because the sets $I^t\cap I^*$, $I^t \setminus I^*$ and $I^* \setminus I^t$ do not intersect with each other
and
$$
(I^t\cap I^*) \cup (I^t  \setminus I^*) \cup (I^* \setminus I^t) = (I^t \cup I^*) \subset (I^{t}_+ \cup I^*) \subset S^{t-1},
$$
and $\frac{\sqrt{5}+3}{2}>1.$
Therefore, the above inequality implies (\ref{UpperBound_Diff_RIP}).

Similar to the rest of the proof of Theorem {\ref{conv_gRIP}},
combining (\ref{Tigher_UpperBound_RIP}) and (\ref{UpperBound_Diff_RIP}),
we can conclude the proof of this theorem.
\end{proof}

\begin{IEEEbiographynophoto}{Yu Wang}
received the B.Sc. degree in Applied Mathematics in 2013 in Xi'an Jiaotong University, Shaanxi, China. From 2007 to 2009, he was a member of the Special Class of the Gifted Young in Xi'an Jiaotong University. From 2009 to 2013, he was in the Science Topnotch Program in China. He is currently pursuing M.Sc. degree in Applied Mathematics in Xi'an Jiaotong University.
\end{IEEEbiographynophoto}

\begin{IEEEbiographynophoto}{Jinshan Zeng}
received the B.S. degree in Information and Computing Sciencises from Xi'an Jiaotong University, Xi'an, in 2008. He is currently pursuing the Ph.D. degree with the School of Mathematics and Statistics, Xi'an Jiaotong University. He has been with the Department of Mathematics, University of California, Los Angeles, as a Visiting Scholar since Nov. 2013. His current research interests include sparse optimization, signal processing and synthetic aperture radar imaging.
\end{IEEEbiographynophoto}

\begin{IEEEbiographynophoto}{Zhimin Peng}
received his B.S. in computational mathematics from Xi'an Jiaotong University in 2011, and then M.A. in applied math from Rice University in 2013.
Now he is a Ph.D. student in the Department of Mathematics at UCLA. His research interests include developing efficient algorithms for anomaly detection and solving large scale convex optimization problems.
\end{IEEEbiographynophoto}

\begin{IEEEbiographynophoto}{Xiangyu Chang}
received the Ph.D. degree in applied mathematics from Xi'an Jiaotong University, China, in 2014.
He is currently an assistant professor at the school of management in Xi'an Jiaotong University, China.
His current research interests include statistical machine learning, high-dimensional statistics, and social network analysis.
\end{IEEEbiographynophoto}

\begin{IEEEbiographynophoto}{Zongben Xu}
received his Ph.D. degree in mathematics from Xi'an Jiaotong University, China, in 1987. He now serves as the Chief Scientist of National Basic Research Program of China (973 Project), and Director of the Institute for Information and System Sciences of the university. He is owner of the National Natural Science Award of China in 2007,and winner of CSIAM Su Buchin Applied Mathematics Prize in 2008. He delivered a 45 minute talk on the International Congress of Mathematicians 2010. He was elected as member of Chinese Academy of Science in 2011. His current research interests include intelligent information processing and applied mathematics.
\end{IEEEbiographynophoto}

\end{document}